\documentclass[12pt]{amsart}%
\usepackage{graphicx}
\usepackage{amscd}
\usepackage{geometry}
\usepackage{tikz-cd}
\usepackage{amsmath}
\usepackage{amsfonts}
\usepackage{amssymb}%
\providecommand{\U}[1]{\protect\rule{.1in}{.1in}}
\newtheorem{theorem}{Theorem}[section]
\theoremstyle{plain}

\newtheorem{question}{Question}

\newtheorem{corollary}[theorem]{Corollary}

\newtheorem{lemma}[theorem]{Lemma}

\newtheorem{proposition}[theorem]{Proposition}

\theoremstyle{definition}
\newtheorem{definition}[theorem]{Definition}
\newtheorem{remark}[theorem]{Remark}
\newtheorem{example}[theorem]{Example}
\numberwithin{equation}{section}

\begin{document}
\title[Coarse $\mathcal{Z}$-boundaries for groups]{Coarse $\mathcal{Z}$-boundaries for groups}
\author{Craig R. Guilbault}
\address{Department of Mathematical Sciences\\
University of Wisconsin-Milwaukee, Milwaukee, WI 53201}
\email{craigg@uwm.edu}
\author{Molly A. Moran}
\address{Department of Mathematics, The Colorado College, Colorado Springs, Colorado 80903}
\email{mmoran@coloradocollege.edu}
\thanks{The authors thank the participants in the UW-Milwaukee Topology Seminar for
useful feedback, with special thanks to Arka Banerjee, Burns Healy, Chris
Hruska, and Boris Okun for detailed comments. Also, many thanks to the referee for their thorough review of the manuscript and helpful suggestions. This research was supported in
part by Simons Foundation Grant 427244, CRG}
\date{January 27, 2021}

\begin{abstract}
We generalize Bestvina's notion of a $\mathcal{Z}$-boundary for a group to
that of a \textquotedblleft coarse $\mathcal{Z}$-boundary.\textquotedblright \space We show that established theorems about $\mathcal{Z}$-boundaries carry over
nicely to the more general theory, and that some wished-for properties of
$\mathcal{Z}$-boundaries become theorems when applied to coarse $\mathcal{Z}%
$-boundaries. Most notably, the property of admitting a coarse $\mathcal{Z}%
$-boundary is a pure quasi-isometry invariant. In the process, we streamline
both new and existing definitions by introducing the notion of a
\textquotedblleft model $\mathcal{Z}$-geometry.\textquotedblright\space In
accordance with the existing theory, we also develop an equivariant version of
the above---that of a \textquotedblleft coarse $E\mathcal{Z}$%
-boundary.\textquotedblright

\end{abstract}
\maketitle

The primary goal of this paper is an expansion of Bestvina's notion of a
$\mathcal{Z}$-boundary. His approach placed Gromov boundaries of torsion-free
hyperbolic groups and visual boundaries of torsion-free CAT(0) groups in a
general framework which allows other classes of groups $G$ to admit a
boundary. Later Dranishnikov relaxed that framework to allow for groups with
torsion. Here we relax the requirements further: instead of a geometric action
of $G$ on a proper metric AR (absolute retract) $X$, we allow for a
\textquotedblleft coarse near-action\textquotedblright\ (a concept that will
be developed here and which contains quasi-actions as special cases). Our
boundaries will be called c$\mathcal{Z}$-boundaries or \textquotedblleft
coarse $\mathcal{Z}$-boundaries.\textquotedblright\space With the new definition in
place, we are able to accomplish a primary goal:

\begin{theorem}
\label{Theorem: primary goal}If a group $G$ admits a c$\mathcal{Z}$-boundary
and $H$ is quasi-isometric to $G$, then $H$ admits a c$\mathcal{Z}$-boundary.
In fact, $G$ and $H$ admit the same c$\mathcal{Z}$-boundaries.
\end{theorem}

\noindent Of course, this generalization is only useful if the information
carried by $\mathcal{Z}$-boundaries is also carried by c$\mathcal{Z}%
$-boundaries. In addition to proving new theorems, we revisit established
results and show that their analogs remain valid in the broader context. Along
the way, we introduce the notion of a \emph{model geometry} and a \emph{model
}$\mathcal{Z}$\emph{-geometry} as part of an effort to streamline the axiom
system for $\mathcal{Z}$-boundaries---both coarse and classical. We also
expand upon the important notion of an $E\mathcal{Z}$-boundary.

\section{Introduction\label{Section: Introduction}}

Bestvina \cite{Bes96} and Dranishnikov \cite{Dra06} developed a general theory
of group boundaries which contains Gromov boundaries of hyperbolic groups and
visual boundaries of CAT(0) groups as special cases. These generalized
boundaries---known as $\mathcal{Z}$\emph{-boundaries}---involve a mix of
geometry, topology, and group theory. The necessary ingredients are:

\begin{itemize}
\item[i)] a geometric action of the group $G$ on some proper metric AR
$\left(  X,d\right)  $; and

\item[ii)] a $\mathcal{Z}$-set compactification $\overline{X}=X\sqcup Z$ in
which compacta from $X$ vanish in $\overline{X}$ as they are pushed toward $Z$
by the $G$-action.
\end{itemize}

\noindent When this can be arranged, $Z$ is called a $\mathcal{Z}%
$\emph{-boundary} and $\left(  \overline{X},Z,d\right)  $ a $\mathcal{Z}%
$\emph{-structure} for $G$. In \cite{FaLa05}, Farrell and Lafont introduced
an additional condition (also satisfied by all hyperbolic and CAT(0)\ groups):

\begin{itemize}
\item[iii)] the $G$-action on $X$ extends to a $G$-action on $\overline{X}$
\end{itemize}

\noindent When all three conditions are satisfied, they call $\left(
\overline{X},Z,d\right)  $ an $E\mathcal{Z}$\emph{-structure}
(\emph{equivariant} $\mathcal{Z}$-structure)\emph{ }and $Z$ an $E\mathcal{Z}%
$\emph{-boundary} for $G$.

Many non-hyperbolic, non-CAT(0) groups, for example, Baumslag-Solitar groups
and systolic groups, are now known to admit $\mathcal{Z}$- or $E\mathcal{Z}%
$-structures. But a full characterization of which groups admit $\left(
E\right)  \mathcal{Z}$-structures is an interesting open problem. Section 6 of
\cite{GuMo19} contains a survey of known results; original sources include
\cite{BeMe91}, \cite{Dah03}, \cite{OsPr09}, \cite{Tir11}, \cite{Mar14},
\cite{Gui14}, \cite{GMT19}, and more recently \cite{EnWu18}, \cite{GuMoSc20},
\cite{CCGHO20},and \cite{Pie18}.\medskip

The following theorem from \cite{GuMo19} has its origins in \cite{Bes96}.

\begin{theorem}
[Generalized Boundary Swapping]\label{Theorem: Generalized Boundary Swapping}%
If $G$ is quasi-isometric to a group $H$ which admits a $\mathcal{Z}$-boundary
$Z$, and $G$ acts geometrically on some proper metric AR $X$, then $G$ also
admits $Z$ as a $\mathcal{Z}$-boundary.
\end{theorem}

It is tempting to conclude that the property of admitting a $\mathcal{Z}%
$-boundary is a quasi-isometry invariant of groups. Unfortunately, that
conclusion would be premature. Consider, for example:\medskip

\noindent\textbf{Question A. }\emph{If }$H\leq G$\emph{ is of finite index,
and }$H$\emph{ admits a }$\mathcal{Z}$\emph{-boundary, does }$G$\emph{ admit a
}$\mathcal{Z}$\emph{-boundary?}\medskip

\noindent The issue is this: Application of Theorem
\ref{Theorem: Generalized Boundary Swapping} requires an AR $X$ on which $G$
acts geometrically. One might hope that the $H$-action on an AR $Y$, implicit
in the hypothesis, can be extended to a $G$-action. Those familiar with an
analogous (open) problem for CAT(0) groups will recognize the difficulty. The
more general problem of finding a geometric action of $G$ on an AR $X$, given
only that $G$ is quasi-isometric to a group $H$ that admits such an action,
appears even more difficult.

In this paper we provide a way around these questions as they pertain to group
boundaries. The key is a relaxation of the main definition to that of a
\emph{coarse} $\mathcal{Z}$\emph{-structure}. Under the new definition, the
requirement of a geometric $G$-action is relaxed to allow for a (proper and
cobounded) \textquotedblleft coarse near-action.\textquotedblright%
\footnote{Coarse near-action is a straightforward generalization of
\emph{quasi-action}. The extra generality is useful when metric spaces are not
necessarily geodesic (or quasi-geodesic). For those who prefer quasi-actions,
nearly all of our definitions and theorems have analogs in that category at
the expense of some mild additional hypotheses. See Section
\ref{Section: Coarse near actions} for details.} When $G$ is quasi-isometric
to a group $H$ which acts geometrically on a proper metric AR $Y$, obtaining a
coarse near-action of $G$ on $Y$ is relatively easy: Use our Theorem
\ref{Theorem: Generalized Svarc-Milnor with converse} (Generalized
\v{S}varc-Milnor) to obtain a coarse equivalence $f:G\rightarrow Y$. Then use
$f$ and a coarse inverse $g:Y\rightarrow G$ to conjugate\ the action of $G$ on
itself to $Y$. The converse portion of Theorem
\ref{Theorem: Generalized Svarc-Milnor with converse} assures that the result
is a coarse near-action on $Y$. When this and the supporting machinery have
been established, we will have:

\begin{theorem}
\label{Theorem: Quasi-isometry invariance for coarse-Z-boundaries}If groups
$G$ and $H$ are quasi-isometric and $\left(  \overline{Y},Z,d\right)  $ is a
$\mathcal{Z}$-structure for $H$, then there is a coarse near-action of $G$ on
$Y$ under which $\left(  \overline{Y},Z,d\right)  $ is a coarse $\mathcal{Z}%
$-structure for $G$.
\end{theorem}

The new approach is useful only if a coarse $\mathcal{Z}$-boundary carries
information, comparable to that of an actual $\mathcal{Z}$-boundary. For
example, work by Bestvina, Mess, Geoghegan, Ontaneda, Dranishnikov, and Roe
\cite{BeMe91}, \cite{Bes96}, \cite{Geo86}, \cite{Ont05}, \cite{GeOn07}%
, \cite{Dra06}, \cite{Roe03} has established that, for $\mathcal{Z}$-structures
on a group $G$:

\begin{itemize}
\item the boundary $Z$ is well-defined up to shape equivalence;

\item $\dim Z$ is a group invariant; and

\item the \v{C}ech cohomology of $Z$ reveals the group cohomology of $G$ with
$RG$-coefficients ($R$ a PID).
\end{itemize}

\noindent We will show that these relationships carry over to the realm of
coarse $\mathcal{Z}$-structures. In fact, the added flexibility allows for a
\emph{strengthening} of some of these conclusions.\bigskip

Here is a quick outline of the remainder of this paper. In Section
\ref{Section: Coarse near actions}, we develop the notion of a \emph{coarse
near-action} of a group $G$ on a metric space $\left(  X,d\right)  $ and prove
some fundamental facts---most significantly, a coarse version of the classical
\v{S}varc-Milnor Lemma, along with a crucial converse. In Section
\ref{Section: Model geometries and model Z-geometries} we introduce the
notions of a \emph{model geometry}\ and a \emph{model $\mathcal{Z}$%
}-\emph{geometry}. Those are used in Section
\ref{Section: Z-structures and coarse Z-structures} to define a coarse
$\mathcal{Z}$-\emph{structure}. That definition includes classical
$\mathcal{Z}$-structures as a special case. In addition, we establish an
equivalent---strikingly simple---formulation of a coarse $\mathcal{Z}%
$-structure which highlights the benefits of our approach. In Sections
\ref{Section: Uniqueness and boundary swapping for cZ-structures} and
\ref{Section: Applications of cZ-structures} we show that key theorems about
groups and their $\mathcal{Z}$-structures remain valid for coarse
$\mathcal{Z}$-structures, and that new, stronger versions are often possible.
In Section \ref{Section EZ-structures and cEZ-structures} we define a
\emph{coarse} $E\mathcal{Z}$\emph{-structure} and prove a few basic theorems.
In a final section, we discuss the fundamental questions (such as Question A
above) which motivated this work. These questions go well beyond $\mathcal{Z}%
$-boundaries; some are well-known to the experts. Our goal is to shine a light
on a family of interesting open problems and to highlight their connections to
the study of group boundaries.

\section{Coarse near-actions\label{Section: Coarse near actions}}

\begin{definition}
\label{Defn: near-action}Let $C\geq0$. A $C$-\emph{near-action} of a group $G$
on a metric space $\left(  X,d\right)  $ is a function $\psi$ which assigns to
each $\gamma\in G$ a function $\psi\left(  \gamma\right)  :X\rightarrow X$
satisfying the following conditions:

\begin{enumerate}
\item $d\left(  \psi\left(  1\right)  \left(  x\right)  ,x\right)  \leq C$,
for all $x\in X$, and

\item $d\left(  \psi\left(  \gamma_{1}\gamma_{2}\right)  \left(  x\right)
,\psi\left(  \gamma_{1}\right)  \psi\left(  \gamma_{2}\right)  \left(
x\right)  \right)  \leq C$, for all $x\in X$ and $\gamma_{1},\gamma_{2}\in G$.
\end{enumerate}

\noindent When the constant $C$ is unimportant, we simply refer to $\psi$ as a
\emph{near-action.}
\end{definition}

\begin{definition}
\label{Defn. coarse equivalence}A function $f:\left(  X,d\right)
\rightarrow\left(  Y,d^{\prime}\right)  $ is

\begin{enumerate}
\item \label{Item 1: defn of coarse equivalence}a \emph{coarse embedding }if
there exist nondecreasing functions $\rho_{-},\rho_{+}:[0,\infty
)\rightarrow\lbrack0,\infty)$ with $\rho_{-}\left(  r\right)  \rightarrow
\infty$ as $r\rightarrow\infty$ such that for all $x_{1},x_{2}\in X$,
\[
\rho_{-}\left(  d\left(  x_{1},x_{2}\right)  \right)  \leq d^{\prime}\left(
f(x_{1}),f(x_{2})\right)  \leq\rho_{+}\left(  d\left(  x_{1},x_{2}\right)
\right)
\]

\item \label{Item 2: defn of coarse equivalence}\emph{coarsely surjective }if
there exists $C\geq0$ such that, for all $y\in Y$, there exists $x\in X$ such
that $d^{\prime}\left(  y,f\left(  x\right)  \right)  \leq C$,

\item \label{Item 3: defn of coarse equivalence}a \emph{coarse equivalence }if
it is a coarsely surjective coarse embedding.
\end{enumerate}

\noindent If $\rho_{-}$, $\rho_{+}$ and $C$ are specified, we sometimes call a
function $f$ which satisfies: (\ref{Item 1: defn of coarse equivalence}) a
$\left(  \rho_{-},\rho_{+}\right)  $-coarse embedding;
(\ref{Item 2: defn of coarse equivalence}) $C$-surjective; and
(\ref{Item 3: defn of coarse equivalence}) a $\left(  \rho_{-},\rho
_{+},C\right)  $-coarse equivalence.
\end{definition}

\begin{remark}
\label{Rk: coarse inverse} 
A coarse embedding $f:(X,d)\rightarrow(Y,d^{\prime})$ is a coarse equivalence
if and only if $f$ has a coarse inverse, that is, there is a coarse embedding
$g:(Y,d^{\prime})\rightarrow(X,d)$ such that $gf$ and $fg$ are boundedly close
to $\operatorname*{id}_{X}$ and $\operatorname*{id}_{Y}$, respectively.
\end{remark}

\begin{definition}
\label{Defn. coarse action}Let $\rho_{-},\rho_{+}:[0,\infty)\rightarrow
\lbrack0,\infty)$ be fixed nondecreasing functions with $\rho_{-}\left(
r\right)  \rightarrow\infty$ as $r\to\infty$ and let $C\geq0$ be a fixed constant. A $\left(
\rho_{-},\rho_{+},C\right)  $\emph{-coarse near-action} is a $C$-near-action
$\psi$ of a group $G$ on a metric space $\left(  X,d\right)  $ with the
property that $\psi\left(  \gamma\right)  $ is a $\left(  \rho_{-},\rho
_{+}\right)  $-coarse embedding for all $\gamma\in G$.
\end{definition}

When the specific control functions and constant are unimportant, we simply
refer to $\psi$ as a \emph{coarse near-action}. Even when not specified, it is
important that a single choice of $\left(  \rho_{-},\rho_{+}\right)  $ applies
uniformly to all $\psi\left(  \gamma\right)  $.

\begin{remark}
Note that properties of a $C$-near-action ensure that each $\psi\left(
\gamma\right)  $ is $2C$-surjective and $\psi\left(  \gamma^{-1}\right)  $ is
a coarse inverse for $\psi\left(  \gamma\right)  $. So each $\psi\left(
\gamma\right)  $ is a coarse equivalence.
\end{remark}

The following generalizations of cocompact and proper actions are useful.

\begin{definition}
Let $\psi$ be a coarse near-action of $G$ on $(X,d)$.

\begin{enumerate}
\item $\psi$ is \emph{cobounded} if there exists $x_{0}\in X$ and $R>0$ such
that for all $y\in X$, there exists $\gamma\in G$ such that $d(\psi
(\gamma)(x_{0}),y)<R$.

\item $\psi$ is \emph{proper} if for each $R>0$, there exists $M\in\mathbb{N}$ such that for
all $x,y\in X$,
\[
\left\vert \{\gamma\in G\ |\ \psi(\gamma)(B_{d}[x,R])\cap B_{d} [y,R]\neq
\emptyset\}\right\vert \leq M
\]
where $B_{d}[x,R]$ is the closed metric ball.
\end{enumerate}
\end{definition}

\begin{remark} As an exercise, one can show that Condition (1) in Definition 2.6 is equivalent to the statement that, for each $x\in X$ there exists an $R'>0$ such that for every $y\in X$, there exists $\gamma\in G$ such that $d(\psi
	(\gamma)(x),y)<R'$. With some additional effort, one can identify a single $R'>0$ which works for all $x$.
\end{remark} 

A geometric action (i.e., proper, cocompact, and by isometries) on a proper
metric space is a proper, cobounded, coarse near-action (with $\rho_{-}%
=\rho_{+}=\operatorname*{id}_{[0,\infty)}$ and $C=0$), so previously studied
$\mathcal{Z}$-structures fall within the new framework. Similarly, Definitions
\ref{Defn. coarse equivalence} and \ref{Defn. coarse action} generalize the
notion of \emph{quasi-isometry }and \emph{quasi-action}, where control
functions are required to be of the form $\frac{1}{K}r-\varepsilon$ and
$Kr+\varepsilon$. (To reconcile the definitions, $\frac{1}{K}r-\varepsilon$
can be truncated below at $0$.) Quasi-actions have been widely studied; see
for example, \cite{Tuk86}, \cite{Tuk94}, \cite{Nek97}, \cite{KlLe01},
\cite{MSW03}, \cite{Man06}, \cite{KlLe09}, \cite{MSW11}, \cite{Eis15}, and
\cite{DrKa18}. In addition, uniform actions by coarse equivalences (called
\emph{coarse actions }in \cite{BDM08}) have been studied. To the best of our
knowledge, the notion of a coarse near-action has not appeared before.

\begin{remark}
\label{Remark: Turning coarse equivalences into quasi-isometries}It is a
standard fact (see, for example, \cite[Cor.1.4.14]{NoYu12}) that a coarse
equivalence between quasi-geodesic spaces is always a quasi-isometry. By the
same reasoning, it can be shown that every coarse action on a quasi-geodesic
space is a quasi-action. Since we do not require our geometric models to be
quasi-geodesic spaces, it is more natural for us to work with coarse actions.
The decision to work in this generality is both historical (see earlier papers
on $\mathcal{Z}$-structures) and practical. Aside from the following minimal
condition, which is immediate for path connected spaces, additional
requirements on $X$ would not lead to stronger conclusions; moreover the
additional effort required to maintain our level of generality is minimal.
\end{remark}

\begin{definition}
\label{Defn. coarsely connected}A metric space $\left(  X,d\right)  $ is
\emph{coarsely connected} if there exists $R>0$ such that, for all
$x,x^{\prime}\in X$, there is a finite sequence $x=x_{0},x_{1},\cdots
,x_{n}=x^{\prime}$ of points (an $R$-chain) in $X$ with $d\left(
x_{i},x_{i+1}\right)  \leq R$ for all $0\leq i\leq n-1$.
\end{definition}

The following proposition can be viewed as a generalization of the classical
\v{S}varc-Milnor Lemma as well as \cite[Cor.0.9]{BDM07} and the reverse
implication of \cite[Th.8.4]{Nek97}. As usual, a finitely generated group is
given the word metric with respect to some finite generating set.

\begin{proposition}
\label{Prop: Coarse Svarc-Milnor with finitely generated conclusion}Suppose a
coarsely connected metric space $\left(  X,d\right)  $ admits a proper,
cobounded coarse near-action $\psi$ by a group $G$. Then $G$ is finitely
generated, and for any $x_{0}\in X$, $\gamma\mapsto\psi\left(  \gamma\right)
\left(  x_{0}\right)  $ is a coarse equivalence between $G$ and $X$.
\end{proposition}

\begin{proof}
Let $\psi$ be a $\left(  \rho_{-},\rho_{+},C\right)  $-coarse near-action on
$\left(  X,d\right)  $ and $x_{0}\in X$. For use later in this proof, note
that we may replace $\rho_{-}$ by an even smaller function (still called
$\rho_{-}$), which is identically $0$ on an interval $[0,B]$, and strictly
increasing to infinity on $[B,\infty)$ for $B>0$.

Choose a single constant $R>0$ satisfying Definition
\ref{Defn. coarsely connected} and such that $\left\{  \psi\left(
\gamma\right)  \left(  x_{0}\right)  \mid\gamma\in G\right\}  $ is $R$-dense
in $X$. Let
\[
S=\left\{  \gamma\in G\mid d\left(  \psi\left(  \gamma\right)  \left(
x_{0}\right)  ,x_{0}\right)  \leq\rho_{+}\left(  3R\right)  +3C\right\}
\]
By properness, $S$ is finite. Let $H\leq G$ be the subgroup generated by $S$.
We claim that $H=G$.

Suppose otherwise. Let
\begin{align*}
A  &  =%
%TCIMACRO{\dbigcup \limits_{g\in H}}%
%BeginExpansion
{\displaystyle\bigcup\limits_{g\in H}}
%EndExpansion
B_{d}\left[  \psi\left(  g\right)  \left(  x_{0}\right)  ,R\right]  \text{,
and}\\
B  &  =%
%TCIMACRO{\dbigcup \limits_{g\in G\backslash H}}%
%BeginExpansion
{\displaystyle\bigcup\limits_{g\in G\backslash H}}
%EndExpansion
B_{d}\left[  \psi\left(  g\right)  \left(  x_{0}\right)  ,R\right]
\end{align*}
Note that $A\cup B=X$, and neither set is empty. So there exists an $R$-chain
connecting a point of $A$ to a point of $B$, and within that chain there exist
points $a\in A$ and $b\in B$ with $d\left(  a,b\right)  \leq R$. Choose
$g_{1}\in H$ and $g_{2}\in G\backslash H$ such that $d\left(  a,\psi\left(
g_{1}\right)  \left(  x_{0}\right)  \right)  \leq R$ and $d\left(
b,\psi\left(  g_{2}\right)  \left(  x_{0}\right)  \right)  \leq R$. Then

\begin{itemize}
\item $d\left(  x_{0},\psi\left(  g_{1}^{-1}\right)  \psi\left(  g_{1}\right)
\left(  x_{0}\right)  \right)  \leq2C$,

\item $d\left(  \psi\left(  g_{1}^{-1}\right)  \psi\left(  g_{1}\right)
\left(  x_{0}\right)  ,\psi\left(  g_{1}^{-1}\right)  \psi\left(
g_{2}\right)  \left(  x_{0}\right)  \right)  \leq\rho_{+}\left(  3R\right)  $, and

\item $d\left(  \psi\left(  g_{1}^{-1}\right)  \psi\left(  g_{2}\right)
\left(  x_{0}\right)  ,\psi(g_{1}^{-1}g_{2})\left(  x_{0}\right)  \right)
\leq C$.
\end{itemize}

By the triangle inequality, $d\left(  x_{0},\psi(g_{1}^{-1}g_{2})\left(
x_{0}\right)  \right)  \leq\rho_{+}\left(  3R\right)  +3C$, so $g_{1}%
^{-1}g_{2}\in S$. But then $g_{2}=g_{1}\left(  g_{1}^{-1}g_{2}\right)  \in H$,
a contradiction. The claim follows.\medskip

Next define $\sigma_{-},\sigma_{+}:[0,\infty)\rightarrow\lbrack0,\infty)$ by%
\begin{align*}
\sigma_{-}(r)  &  =\inf\left\{  d\left(  x_{0},\psi\left(  \gamma\right)
\left(  x_{0}\right)  \right)  \mid d\left(  1,\gamma\right)  \geq r\right\}
\text{, and }\\
\sigma_{+}(r)  &  =\sup\left\{  d\left(  x_{0},\psi\left(  \gamma\right)
\left(  x_{0}\right)  \right)  \mid d\left(  1,\gamma\right)  \leq r\right\}
\end{align*}
Notice that both $\sigma_{-}$ and $\sigma_{+}$ are nondecreasing, $\sigma
_{-}(r)\leq\sigma_{+}(r)$ for all $r$, and (by properness), $\sigma_{-}\left(
r\right)  \rightarrow\infty$ as $r\rightarrow\infty$.

Let $\gamma_{1},\gamma_{2}\in G$. Then
\begin{align*}
\rho_{-}\left(  d\left(  \psi\left(  \gamma_{1}\right)  \left(  x_{0}\right)
,\psi\left(  \gamma_{2}\right)  \left(  x_{0}\right)  \right)  \right)   &
\leq d\left(  \psi\left(  \gamma_{1}^{-1}\right)  \psi\left(  \gamma
_{1}\right)  \left(  x_{0}\right)  ,\psi\left(  \gamma_{1}^{-1}\right)
\psi\left(  \gamma_{2}\right)  \left(  x_{0}\right)  \right) \\
&  \leq d\left(  \psi\left(  \gamma_{1}^{-1}\right)  \psi\left(  \gamma
_{1}\right)  \left(  x_{0}\right)  ,x_{0})+d(x_{0},\psi\left(  \gamma_{1}%
^{-1}\right)  \psi\left(  \gamma_{2}\right)  \left(  x_{0}\right)  \right)
)\\
&  \leq2C+[C+d\left(  x_{0},\psi\left(  \gamma_{1}^{-1}\gamma_{2}\right)
\left(  x_{0}\right)  \right)  ]\\
&  \leq3C+\sigma_{+}\left(  d\left(  \gamma_{1},\gamma_{2}\right)  \right)
\end{align*}

Therefore
\begin{equation}
d\left(  \psi\left(  \gamma_{1}\right)  \left(  x_{0}\right)  ,\psi\left(
\gamma_{2}\right)  \left(  x_{0}\right)  \right)  \leq\rho_{-}^{-1}\left(
3C+\sigma_{+}\left(  d\left(  \gamma_{1},\gamma_{2}\right)  \right)  \right)
\end{equation}
where $\rho_{-}^{-1}:[0,\infty)\rightarrow\lbrack B,\infty)$ is \emph{defined}
to be the inverse of $\left.  \rho_{-}\right\vert _{[B,\infty)}$. Since
$\rho_{-}^{-1}$ strictly increases to infinity, we may define $\tau_{+}%
:[0,\infty)\rightarrow\lbrack0,\infty)$ by
\[
\tau_{+}\left(  r\right)  =\rho_{-}^{-1}\left(  3C+\sigma_{+}(r)\right)
\]
to obtain an upper control function satisfying the inequality
\[
d\left(  \psi\left(  \gamma_{1}\right)  \left(  x_{0}\right)  ,\psi\left(
\gamma_{2}\right)  \left(  x_{0}\right)  \right)  \leq\tau_{+}\left(  d\left(
\gamma_{1},\gamma_{2}\right)  \right)
\]

As for the lower control function, note that%
\begin{align*}
\rho_{-}\sigma_{-}(d\left(  \gamma_{1},\gamma_{2}\right)  )  &  =\rho
_{-}\sigma_{-}(d\left(  1,\gamma_{1}^{-1}\gamma_{2}\right)  )\\
&  \leq\rho_{-}(d(x_{0},\psi\left(  \gamma_{1}^{-1}\gamma_{2}\right)  \left(
x_{0}\right)  ))\\
&  \leq d(\psi\left(  \gamma_{1}\right)  \left(  x_{0}\right)  ,\psi\left(
\gamma_{1}\right)  \psi\left(  \gamma_{1}^{-1}\gamma_{2}\right)  \left(
x_{0}\right)  )\\
&  \leq d\left(  \psi\left(  \gamma_{1}\right)  \left(  x_{0}\right)
,\psi\left(  \gamma_{2}\right)  \left(  x_{0}\right)  \right)  +d\left(
\psi\left(  \gamma_{2}\right)  \left(  x_{0}\right)  ,\psi\left(  \gamma
_{1}\right)  \psi\left(  \gamma_{1}^{-1}\gamma_{2}\right)  \left(
x_{0}\right)  \right) \\
&  \leq d\left(  \psi\left(  \gamma_{1}\right)  \left(  x_{0}\right)
,\psi\left(  \gamma_{2}\right)  \left(  x_{0}\right)  \right)  +C
\end{align*}

Therefore
\begin{equation}
\rho_{-}\sigma_{-}(d\left(  \gamma_{1},\gamma_{2}\right)  )-C\leq d\left(
\psi\left(  \gamma_{1}\right)  \left(  x_{0}\right)  ,\psi\left(  \gamma
_{2}\right)  \left(  x_{0}\right)  \right)
\end{equation}

Define $\tau_{-}:[0,\infty)\rightarrow\lbrack0,\infty)$ by
\[
\tau_{-}\left(  r\right)  =\left\{
\begin{array}
[c]{cc}%
0 & \text{if }\rho_{-}\sigma_{-}(r)-C<0\\
\rho_{-}\sigma_{-}(r)-C & \text{otherwise}%
\end{array}
\right.
\]
Since the right-hand side of (2.2) is $\geq0$, the inequality
\[
\tau_{-}\left(  d\left(  \gamma_{1},\gamma_{2}\right)  \right)  \leq d\left(
\psi\left(  \gamma_{1}\right)  \left(  x_{0}\right)  ,\psi\left(  \gamma
_{2}\right)  \left(  x_{0}\right)  \right)
\]
holds.
\end{proof}

The next proposition generalizes \cite[Appendix A]{Eis15} and the forward
implication of \cite[Th.8.4]{Nek97}.

\begin{proposition}
If a metric space $\left(  X,d\right)  $ is coarsely equivalent to a
countable\footnote{Countability is needed to endow $G$ with a proper word
length metric, well-defined up to coarse equivalence. See \cite[\S 1.2]%
{NoYu12}.} group $G$ then $X$ admits a proper cobounded coarse near-action by
$G$.
\end{proposition}

\begin{proof}
Let $f_{1}:G\rightarrow X$ be a $(\rho_{-},\rho_{+},C)$-coarse equivalence,
where, as in the proof of Proposition \ref{Prop: Coarse Svarc-Milnor with finitely generated conclusion}, $\rho_{-}$ is identically $0$ on
some interval $[0,B]$ and strictly increasing on $[B,\infty]$; so $\rho
_{-}^{-1}$ is defined on $[B,\infty)$. Let $f_{2}:X\rightarrow G$ be a coarse inverse of $f_1$. Then $f_2$ is a 
$\left(  \sigma_{-},\sigma_{+}\right)  $-coarse embedding with the property that
$d(x,f_{1}f_{2}(x))\leq C$ for all $x\in X$ and $d(\gamma,f_{2}f_{1}%
(\gamma))\leq C$ for all $\gamma\in G$. We will show that $\psi(\gamma
):X\rightarrow X$, defined by $\psi(\gamma)=f_{1}\circ\gamma\circ f_{2}$
(where $\gamma$ simultaneously represents an element of $G$ and the isometry
of $G$ defined by left multiplication by $\gamma$) determines a coarse
near-action of $G$ on $X$.

First note that, for each $\gamma\in G$, $f_{1}\circ\gamma\circ f_{2}$ is a
$\left(  \rho_{-}\circ\sigma_{-},\rho_{+}\circ\sigma_{+}\right)  $-coarse
embedding. Thus, to show that $\psi$ is a coarse near-action, we only need to
check the conditions of Definition \ref{Defn. coarse action}.

We observe first that $\psi(1)=f_{1}f_{2}$, which is of bounded distance $\leq C$
from $\operatorname*{id}_{X}$. Next let $\gamma_{1},\gamma_{2}\in G$ and $x\in
X$. By the near inversive properties of $f_{1}$ and $f_{2}$,
\[
d\left(  f_{2}f_{1}\left(  \gamma_{2} f_{2}\left(  x\right)  \right)
,\gamma_{2}f_{2}\left(  x\right)  \right)  \leq C
\]
Left-multiplication by $\gamma_{1}$ yields
\[
d\left(  \gamma_{1}f_{2}f_{1}\left(  \gamma_{2}f_{2}\left(  x\right)  \right)
,\gamma_{1}\gamma_{2}f_{2}\left(  x\right)  \right)  \leq C
\]
Plugging both terms into $f_{1}$ gives
\[
d\left(  f_{1}\gamma_{1}f_{2}f_{1}\left(  \gamma_{2}f_{2}\left(  x\right)
\right)  ,f_{1}\gamma_{1}\gamma_{2}f_{2}\left(  x\right)  \right)  \leq
\rho_{+}(C)
\]

In other words,
\[
d(\psi(\gamma_{1})\psi(\gamma_{2})(x),\psi(\gamma_{1}\gamma_{2})(x))\leq
\rho_{+}(C)
\]
Letting $C^{\prime}=\max\left\{  C,\rho_{+}(C)\right\}  $, we have shown that
$\psi$ is a $\left(  \rho_{-}\circ\sigma_{-},\rho_{+}\circ\sigma_{+}
,C^{\prime}\right)  $-coarse near-action of $G$ on $X$.

Now, we show that this coarse near-action $\psi$ of $G$ on $X$ is proper and cobounded.

\noindent First, observe that for each $\gamma\in G$,
\begin{align*}
d(f_{1}(\gamma), \psi(\gamma)(f_{1}(1)))  &  =d(f_{1}%
(\gamma), f_{1}\gamma f_{2}(f_{1}(1)))\\
&  \leq\rho_{+}(d(\gamma, \gamma f_{2}(f_{1}(1))))\\
&  = \rho_{+}(d(1, f_{2}(f_{1}(1))))\\
&  \leq\rho_{+}(C)\\
&  \leq C^{\prime}%
\end{align*}

Set $x_0=f_1(1)$ and $R=C'+C$. Let $y\in X$. Since $f_{1}$ is a coarse equivalence, choose $\gamma\in G$
such that $d(f_{1}(\gamma), y)\leq C$.  Then,
\begin{align*}
&  d(\psi(\gamma)(x_0),y)\\
& =d(\psi(\gamma)(f_1(1)),y)\\
&  \leq d(\psi(\gamma)(f_{1}(1)), f_{1}(\gamma))+d(f_{1}(\gamma), y)\\
&  \leq C^{\prime}+C
\end{align*}
proving that $\psi$ is cobounded. 

To show that $\psi$ is proper, let $R>0$ and
$x,y\in X$. We wish to show that
\[
\{\gamma\in G|\psi(\gamma)(B[x,R])\cap B[y,R]\neq\emptyset\}\subset\{\gamma\in
G|\gamma(B[\gamma_{1},M])\cap B[\gamma_{2},M]\neq\emptyset\}
\]
for some $\gamma_{1},\gamma_{2}\in G$ and $M=\rho_{-}^{-1}(2C^{\prime}%
+R+\rho_{+}\sigma_{+}(C^{\prime}))$. Since the latter set is finite ($G$ acts
properly on itself), we can conclude that $\psi$ is proper.

Thus, choose a $\gamma\in G$ so that $\psi(\gamma)(B[x,R])\cap B[y,R]\neq
\emptyset$. Let $z\in B[x,R]$ so that $\psi(\gamma)(z)\in B[y,R]$. Since
$f_{1}$ is a coarse equivalence, there is $\gamma_{1},\gamma_{2},\gamma_{3}\in
G$ so that $d(f_{1}(\gamma_{1}),x)\leq C^{\prime}$, $d(f_{1}(\gamma
_{2}),y)\leq C^{\prime}$ and $d(f_{1}(\gamma_{3}), z)\leq C^{\prime}$. From
this we observe that%

\[
d(\psi(\gamma)(z), \psi(\gamma)(f_{1}(\gamma_{3})))\leq\rho_{+}\sigma
_{+}(C^{\prime})
\]
Moreover, we have:
\begin{align*}
&  d(f_{1}(\gamma_{2}), f_{1}(\gamma\gamma_{3}))\\
&  \leq d(f_{1}(\gamma_{2}), y)+d(y,\psi(\gamma)(z))+d(\psi(\gamma)(z),
\psi(\gamma)(f_{1}(\gamma_{3})))+d(\psi(\gamma)(f_{1}(\gamma_{3})),
f_{1}(\gamma\gamma_{3}))\\
&  \leq2C^{\prime}+R+\rho_{+}\sigma_{+}(C^{\prime})
\end{align*}
Since $f_{1}$ is a coarse-equivalence, using the previous inequality, we have
\begin{align*}
\rho_{-}(d(\gamma_{2}, \gamma\gamma_{3})  &  \leq2C^{\prime}+R+\rho_{+}%
\sigma_{+}(C^{\prime})
\end{align*}
Hence $d(\gamma_{2},\gamma\gamma_{3})\leq\rho_{-}^{-1}(2C^{\prime}+R+\rho
_{+}\sigma_{+}(C^{\prime}))\leq M$, proving $\gamma\gamma_{3}\in B[\gamma
_{2},M]$. Furthermore, we know
\begin{align*}
d(f_{1}(\gamma_{1}), f_{1}(\gamma_{3}))  &  \leq d(f_{1}(\gamma_{1}),x)+d(x,
z)+d(z, f_1(\gamma_{3}))\\
&  \leq2C^{\prime}+R
\end{align*}
Once again, using the previous inequality we see that $d(\gamma_{1},\gamma
_{3})\leq\rho_{-}^{-1}(2C^{\prime}+R)\leq M$ proving that $\gamma\gamma_{3}%
\in\gamma(B[\gamma_{1},M])$. Thus, we have $\gamma\gamma_{3}\in\{\gamma\in
G|\gamma(B[\gamma_{1},M])\cap B[\gamma_{2},M]\neq\emptyset\}$.
\end{proof}

Combining the previous two propositions and restricting to the cases of
primary interest (for our purposes) gives us the following.

\begin{theorem}
[Generalized \v{S}varc-Milnor with converse]%
\label{Theorem: Generalized Svarc-Milnor with converse}Let $\left(
X,d\right)  $ be a coarsely connected metric space and $G$ a finitely
generated group. Then $X$ admits a proper, cobounded coarse near-action by $G$
if and only if $X$ is coarsely equivalent to $G$.
\end{theorem}

\begin{corollary}
A quasi-geodesic metric space $\left(  X,d\right)  $ admits a proper,
cobounded quasi-action by $G$ if and only if $G$ is finitely generated and
quasi-isometric to $X$.
\end{corollary}

\begin{proof}
This is \cite[Th.8.4]{Nek97}. Alternatively, use Theorem
\ref{Theorem: Generalized Svarc-Milnor with converse} and Remark
\ref{Remark: Turning coarse equivalences into quasi-isometries}.
\end{proof}

\section{Model geometries and model $\mathcal{Z}$%
-geometries\label{Section: Model geometries and model Z-geometries}}

Throughout this paper, all spaces are assumed separable and metrizable. A
metric space $\left(  X,d\right)  $ is \emph{proper} if every closed metric
ball $B_{d}\left[  x,r\right]  \subseteq X$ is compact. It is \emph{cocompact}
if there exist $x_{0}\in X$ and $R>0$ so that $\left\{  \left.  B_{d}\left[
\gamma x_{0},R\right]  \,\right\vert \,\gamma\in\operatorname{Isom}\left(
X\right)  \right\}  $ covers $X$. Here $\operatorname{Isom}\left(  X\right)  $
denotes the group of self-isometries of $X$. For later use, we review a few
well-known properties that follow from properness and/or cocompactness.

A metric space $\left(  X,d\right)  $ is \emph{uniformly contractible }if for
each $R>0$, there exists $S>R$ so that every open metric ball $B_{d}(x,R)$
contracts in $B_{d}(x,S)$.

\begin{lemma}
\label{Lemma: uniform contractibility}If a proper metric space $\left(
X,d\right)  $ is cocompact and contractible, then it is uniformly contractible.
\end{lemma}

\begin{proof}
See, for example, \cite[Lemma 4.8]{GuMo19}.
\end{proof}

An open cover $\mathcal{U}$ of a metric space $\left(  X,d\right)  $ is
\emph{uniformly bounded} if $\left\{  \operatorname*{diam}\left(  U\right)
\mid U\in\mathcal{U}\right\}  $ is bounded above; in that case the supremum of
this set is the \emph{mesh} of $\mathcal{U}$. The \emph{order} of
$\mathcal{U}$ is the largest $k\in\mathbb{N\cup
}\left\{  \infty\right\}  $ such that some $x\in X$ is contained in $k$
elements of $\mathcal{U}$. We say that $\left(  X,d\right)  $ has
\emph{macroscopic dimension }$\leq n$ if there exists a uniformly bounded open cover of
$X$ having order $\leq n+1$. A cover $\mathcal{V}$ \emph{refines}
$\mathcal{U}$ if, for every $V\in\mathcal{V}$ there exists some $U\in
\mathcal{U}$ such that $V\subseteq U$; in that case we write $\mathcal{U\succ
V}$.

\begin{lemma}
\label{Lemma: finite macroscopic dimension and refinements}Let $\left(
X,d\right)  $ be proper and cocompact. Then $X$ has finite macroscopic
dimension; in fact, $X$ admits a sequence of finite order, uniformly bounded,
open covers $\mathcal{U}_{0}\succ\mathcal{U}_{1}\succ\mathcal{U}_{2}%
\succ\cdots$ such that $\operatorname{mesh}\left(  \mathcal{U}_{i}\right)
\rightarrow0$.
\end{lemma}

\begin{proof}
The argument provided in \cite[Lemma 3.1]{Mor16a} proves the existence of some
$x_{0}\in X$, $r>0$, and $\left\{  \lambda_{i}\right\}  _{i\in\mathbb{N}%
}\subseteq$ $\operatorname{Isom}\left(  X\right)  $ such that $\mathcal{U}%
_{0}=\left\{  B_{d}\left(  \lambda_{i}(x_{0}),r\right)  \right\}  $ is a
finite order cover of $X$. The following method produces a refinement
$\mathcal{V}_{\varepsilon}$ of $\mathcal{U}_{0}$ of mesh $\leq\varepsilon$ for
arbitrary $\varepsilon>0$: Use properness to choose a finite cover
$\mathcal{W}_{\varepsilon}^{\prime}$ of $B_{d}[x_{0},r]$ by open $\varepsilon
$-balls, and let $\mathcal{W}_{\varepsilon}$ be the collection of
intersections of the elements of $\mathcal{W}_{\varepsilon}^{\prime}$ with
$B_{d}\left(  x_{0},r\right)  $. Then let
\[
\mathcal{V}_{\varepsilon}=\left\{  \lambda_{i}\left(  W\right)  \mid
W\in\mathcal{W}_{\varepsilon}\text{ and }i\in\mathbb{N}\right\}
\]
Since $\mathcal{W}_{\varepsilon}$ is finite and $\mathcal{U}_{0}$ has finite
order, $\mathcal{V}_{\varepsilon}$ has finite order.

To produce the sequence of open covers promised in the lemma, we apply the
above procedure inductively. Let $\varepsilon_{1}$ be arbitrary and
$\mathcal{U}_{1}=\mathcal{V}_{\varepsilon_{1}}$. Next let $\varepsilon_{2}>0$
be a Lebesgue number for the cover $\mathcal{W}_{\varepsilon_{1}}^{\prime}$ to
obtain $\mathcal{U}_{2}=\mathcal{V}_{\varepsilon_{2}}$ which refines
$\mathcal{U}_{1}$ and has mesh $\leq\varepsilon_{2}$. Continue inductively,
making sure that the chosen Lebesgue numbers converge to $0$.
\end{proof}

\begin{lemma}
\label{Lemma: bounded homotopy}Suppose $(X,d)$ is uniformly contractible;
$\mathcal{U}$ is a uniformly bounded, finite order, open cover of $X$; $K$ is
the nerve of $\mathcal{U}$; and $\psi:X\rightarrow K$ is a corresponding
barycentric map. Then there is a map $s:K\rightarrow X$ and a bounded homotopy
$H:X\times I\rightarrow X$ joining $s\circ\psi$ with the identity
$\operatorname*{id}_{X}$.
\end{lemma}

\begin{proof}
Since $K$ is finite-dimensional and $X$ is uniformly contractible, it is
straightforward to build a map $s:K\rightarrow X$ inductively over the skeleta
of $K$ such that $s\circ\psi$ is bounded distance from $\operatorname*{id}%
_{X}$. (This is an easier version of \cite[Prop.5.2]{GuMo19}.) From there one
can apply \cite[Cor.5.3]{GuMo19} to obtain the desired homotopy.
\end{proof}

A locally compact space $X$ is an \emph{absolute neighborhood retract}
(\emph{ANR} for short) if, whenever $X$ is embedded as a closed subset of
another space $Y$, some neighborhood of $X$ retracts onto $X$. A contractible
ANR is called an \emph{absolute retract} or simply an \emph{AR}. The category
of ANRs provides a common ground of \textquotedblleft nice\textquotedblright%
\ spaces which includes manifolds, locally finite complexes, and proper CAT(0)
spaces---the spaces most commonly encountered in geometric group theory. For a
quick introduction to ANRs, see \cite[\S 2]{GuMo19}. It is worth noting that
some authors do not require ANRs to be locally compact (or separable and
metrizable). For our purposes, we consider those conditions to be part of the
definition. We use the term \emph{metric AR }(or \emph{metric ANR}) when a
specific metric plays a role.

\begin{definition}
A \emph{model geometry} is a proper, cocompact, metric AR $\left(  X,d\right)
$.
\end{definition}

Given a model geometry, we often seek a nice compactification. A closed subset
$A$ of a space $Y$, is a \emph{$\mathcal{Z}$-set} if there exists a homotopy
$H:Y\times\lbrack0,1]\rightarrow Y$ such that $H_{0}=\operatorname*{id}_{Y}$
and $H_{t}(Y)\subset Y-A$ for every $t>0$. In this case we say $H$
\emph{instantly homotopes }$Y$ \emph{off from} $A$. A \emph{$\mathcal{Z}%
$-com\-pact\-i\-fi\-ca\-tion} of a space $X$ is a com\-pact\-i\-fi\-ca\-tion
$\overline{X}=X\sqcup Z$ such that $Z$ is a $\mathcal{Z}$-set in $\overline
{X}$. If $X$ is separable and metrizable then so is $\overline{X}$; and if $X$
is an AR then so is $\overline{X}$. For these and other facts about
\emph{$\mathcal{Z}$}-com\-pact\-i\-fi\-ca\-tions, see \cite[\S 3]{GuMo19}.

A \emph{controlled }$\mathcal{Z}$\emph{-com\-pact\-i\-fi\-ca\-tion }of a
proper metric space $\left(  X,d\right)  $ is a $\mathcal{Z}$%
-com\-pact\-i\-fi\-ca\-tion $\overline{X}$ satisfying the additional condition:%

\begin{gather*}
\text{(\ddag) For\ every}\ R>0\ \text{and\ open\ cover\ }\mathcal{U}%
\ \text{of}\ \overline{X}\text{,\ there\ is\ a\ compact\ set\ }C\subset X\ \\
\text{so\ that\ if\ }A\subseteq X-C\ \text{and\ }\operatorname*{diam}%
\nolimits_{d}A<R\text{,\ then}\ A\subseteq U\ \text{for\ some\ }%
U\in\mathcal{U}.
\end{gather*}

\noindent If we choose a metric $\overline{d}$ for $\overline{X}$, condition
$\emph{(\ddag)}$ is equivalent to:\medskip%
\begin{align*}
\text{(\ddag}^{\prime}\text{) For\ every\ }R  &  >0\text{\ and\ }%
\epsilon>0\text{, there\ is\ a\ compact\ set\ }C\subset X\ \text{so\ that\ }\\
\text{if\ }A  &  \subseteq X-C\text{\ and\ }\operatorname*{diam}%
\nolimits_{d}A<R\text{,\ then\ }\operatorname*{diam}\nolimits_{\overline{d}%
}A<\epsilon\text{.}%
\end{align*}

\begin{definition}
A \emph{model $\mathcal{Z}$-geometry} $\left(  \overline{X},Z,d\right)  $ is a
controlled \emph{$\mathcal{Z}$}-com\-pact\-i\-fi\-ca\-tion $\overline
{X}=X\sqcup Z$ of a model geometry $\left(  X,d\right)  $. In this case we
refer to $\left(  X,d\right)  $ as the \emph{underlying geometry} and the
space $Z$ as the \emph{$\mathcal{Z}$-boundary} of $\left(  \overline
{X},Z,d\right)  $.
\end{definition}

It is important to note that the space $\overline{X}=X\sqcup Z$ does not, by
itself, determine the model \emph{$\mathcal{Z}$}-geometry; the metric $d$ is a
crucial ingredient.\footnote{By contrast, the choice of a metric $\overline
{d}$ realizing the topology on $\overline{X}$, while convenient, is of little
additional significance; any such metric will do.} It is also worth noting
that a given model geometry $\left(  X,d\right)  $ can admit any number of
distinct model \emph{$\mathcal{Z}$}-geometries (possibly none at all). It is
useful to think about the following simple examples.

\begin{example}
\label{Example: planar geometries}The Euclidean plane $\left(
%TCIMACRO{\U{211d} }%
%BeginExpansion
\mathbb{R}
%EndExpansion
^{2},d_{E}\right)  $ and hyperbolic plane $\left(  \mathbb{H}^{2}%
,d_{H}\right)  $ are model geometries. Adding the standard visual circles at
infinity gives model \emph{$\mathcal{Z}$}-geometries $\left(  \overline{%
%TCIMACRO{\U{211d} }%
%BeginExpansion
\mathbb{R}
%EndExpansion
^{2}},S^{1},d_{E}\right)  $ and $\left(  \overline{\mathbb{H}^{2}},S^{1}%
,d_{H}\right)  $. In each case, all isometries of the original space extend to
homeomorphisms of the compactification. By contrast, if we quotient out the
upper half-circle in either of these boundaries, we get a new model
\emph{$\mathcal{Z}$}-geometry for which the boundary is still a circle, but
now many isometries do not extend. We will return to the issue of
extendability in Section \ref{Section EZ-structures and cEZ-structures}.
\end{example}

\begin{example}
\label{Example: n-dimensional geometries}As above, we can obtain model
\emph{$\mathcal{Z}$}-geometries $\left(  \overline{%
%TCIMACRO{\U{211d} }%
%BeginExpansion
\mathbb{R}
%EndExpansion
^{n}},S^{n-1},d_{E}\right)  $ and \newline$\left(  \overline{\mathbb{H}^{n}%
},S^{n-1},d_{H}\right)  $ by adding the visual $\left(  n-1\right)  $-sphere
at infinity to Euclidean and hyperbolic $n$-space. If $A\subseteq S^{n-1}$ is
a non-cellular cell-like set (such as a Fox-Artin arc or the Whitehead
continuum in $S^{3}$), then quotienting out by $A$ produces model
\emph{$\mathcal{Z}$}-geometries with boundaries not homeomorphic to $S^{n-1}$.
\end{example}

We require cocompactness in our model geometries but we do not assume they
admit a geometric group action---or even proper, cobounded coarse near action.
Heintze \cite{Hei74} has observed the existence of homogeneous negatively
curved Riemannian manifolds which, by virtue of not being symmetric spaces, do
not admit quotients of finite volume. Cornulier \cite{Cor18} shows that these
spaces are not even quasi-isometric to a finitely generated group, and hence
admit no proper, cobounded coarse near action. More recently, Healy and
Pengitore \cite{HePe20} have constructed higher rank CAT(0) spaces with similar
properties. As such, there are model geometries and $\mathcal{Z}$-geometries
not relevant to the group theoretic applications that are the focus of this
paper. Nevertheless, several theorems in the coming sections \emph{can} be
applied to those spaces.\medskip

Model geometries are not required to be finite-dimensional but the argument
presented in \cite{Mor16a}, or more explicitly \cite[Th.2]{GuMo16}, gives the following.

\begin{theorem}
\label{Theorem: Finite-dimensionality of Z-boundaries}The \emph{$\mathcal{Z}%
$-boundary} of every \emph{model $\mathcal{Z}$-geometry} $\left(  \overline
{X},Z,d\right)  $ has finite Lebesgue covering dimension. More specifically,
if the macroscopic dimension of $X$ is $n$, then $\dim Z\leq n-1$.
\end{theorem}

We close this section with an observation that will be used in Section
\ref{Section: Applications of cZ-structures} and can be found in
\cite{GuMo19}. We repeat it here for easy access and to emphasize the
difference between the metrics $d$ and $\overline{d}$ mentioned above.

\begin{lemma}
Suppose $(\overline{X},Z,d)$ is a model $\mathcal{Z}$-geometry. For each $z\in
Z$, neighborhood $\overline{U}$ of $z$ in $\overline{X}$ and $r>0$, there is a
neighborhood $\overline{V}$ of $z$ in $\overline{X}$ such that $d(V,X-U)\geq
r$ where $V=\overline{V}-Z$ and $U=\overline{U}-Z$.
\end{lemma}

\section{Defining $\mathcal{Z}$-structures and coarse $\mathcal{Z}%
$-structures\label{Section: Z-structures and coarse Z-structures}}

We can now formulate one of our main definitions---that of a \textquotedblleft
coarse $\mathcal{Z}$-structure\textquotedblright. The task is made simpler by
using the notion of a model $\mathcal{Z}$\emph{-}geometry. First we
reformulate the classical notion of a $\mathcal{Z}$-structure in this way (see \cite[Definition 1.1]{Bes96}, \cite[Definition 1]{Dra06}, or \cite[Definition 6.1]{GMT19} for versions of this classical definition).
Lemma 6.4 of \cite{GuMo19} assures that this formulation is equivalent to the original.

\begin{definition}
\label{Defn. Z-structure}A $\mathcal{Z}$\emph{-structure} on a group $G$
consists of a model $\mathcal{Z}$-geometry $\left(  \overline{X},Z,d\right)  $
and a geometric action of $G$ on $\left(  X,d\right)  $. In this case, we call
$Z$ a $\mathcal{Z}$\emph{-boundary }for $G$.
\end{definition}

\begin{remark}
Equivalently, a $\mathcal{Z}$\emph{-structure} on $G$ is a homomorphism
$\phi:G\rightarrow\operatorname{Isom}\left(  X\right)  $ such that $\ker\phi$
is finite and $\phi\left(  G\right)  $ is both cocompact and proper.
%If
%$\phi\left(  G\right)  \leq\mathcal{U}_{\overline{X}}$, we call this an
%$\mathcal{EZ}$\emph{-structure} and $Z$ an $\mathcal{EZ}$\emph{-boundary for
%}$G$.

\end{remark}

We are now ready to generalize.

\begin{definition}
\label{Defn. coarse-Z-structure}A \emph{coarse} $\mathcal{Z}$\emph{-structure}
on a group $G$ (\emph{c}$\mathcal{Z}$\emph{-structure}\ for short) consists of
a model $\mathcal{Z}$-geometry $\left(  \overline{X},Z,d\right)  $ and a
proper, cobounded, coarse near-action of $G$ on $X$. In this case we call $Z$
a coarse $\mathcal{Z}$\emph{-boundary} (or \emph{c}$\mathcal{Z}$%
\emph{-boundary}) for $G$.
\end{definition}

The \textquotedblleft if and only if\textquotedblright\ nature of Theorem
\ref{Theorem: Generalized Svarc-Milnor with converse} allows for a simple
equivalent definition.

\begin{definition}
[alternative formulation]%
\label{Defn: Alternate characterization of cZ-structure}A \emph{c}%
$\mathcal{Z}$\emph{-structure} on a finitely generated group $G$ consists of a
model $\mathcal{Z}$-geometry $\left(  \overline{X},Z,d\right)  $ and a coarse
equivalence $f:G\rightarrow\left(  X,d\right)  $. \medskip
\end{definition}

\begin{remark}
Proposition \ref{Prop: Coarse Svarc-Milnor with finitely generated conclusion}
implies that a group $G$ admitting a c$\mathcal{Z}$-structure is finitely
generated. In that context, we always give $G$ a standard word length metric.
\end{remark}

Since every $\mathcal{Z}$-structure is a $c\mathcal{Z}$-structure, we
immediately have many groups that admit $c\mathcal{Z}$-structures. As for new
examples, those are produced primarily by applications of Theorem
\ref{Theorem: primary goal}. We will look more closely at specific cases in
the next section.

For those who prefer working with quasi-isometries and quasi-actions, we
formulate a definition in that category. The situation is complicated slightly
by the extra hypothesis required in the \v{S}varc-Milnor Theorem for
quasi-actions (see Remark \ref{Remark: Turning coarse equivalences into quasi-isometries}).

\begin{definition}
\label{Defn. quasi-Z-structure}A \emph{quasi-}$\mathcal{Z}$\emph{-structure}
on a group $G$ (\emph{q}$\mathcal{Z}$\emph{-structure}\ for short) consists of
a model $\mathcal{Z}$-geometry $\left(  \overline{X},Z\right)  $ and a proper,
cobounded quasi-action of $G$ on $X$. In that case we call $Z$ a
\emph{quasi-}$\mathcal{Z}$\emph{-boundary} (or q$\mathcal{Z}$\emph{-}boundary)
for $G$.
\end{definition}

\begin{proposition}
\label{Prop: Alternate characterization of qZ-structure}A finitely generated
group $G$ admits a q$\mathcal{Z}$-structure based on a model $\mathcal{Z}%
$-geometry $\left(  \overline{X},Z,d\right)  $ if $G$ is quasi-isometric to
$\left(  X,d\right)  $. If $\left(  X,d\right)  $ is quasi-geodesic space the
converse is true.
\end{proposition}

\begin{proof}
Instead of Theorem \ref{Theorem: Generalized Svarc-Milnor with converse},
apply Theorem 8.4 of \cite{Nek97}.
\end{proof}

For the sake of simplicity, we will focus primarily on the coarse category in
the remainder of this paper.\bigskip

Before moving on, we discuss the general class of groups that are candidates
for $\mathcal{Z}$- and $c\mathcal{Z}$-structures. It is well-known that, for a
torsion-free group $G$ to admit a $\mathcal{Z}$-structure, it must be of
\emph{Type F}, meaning that there exists a finite $K(G,1)$ complex. This is
true for the following reason: Since $G$ is torsion-free, the assumed proper,
cocompact action on an AR, $X$, is free; so the quotient map $q:X\rightarrow
G\backslash X$ is a covering map. As a result, $G\backslash X$ is a compact
aspherical ANR. A theorem of West \cite{Wes77} assures that $G\backslash X$ is
homotopy equivalent to a finite CW complex $K$, which is our $K(G,1)$ complex.
By passing to universal covers, we can give an alternative definition: $G$ is
Type F if there exists a contractible CW complex admitting a proper, free,
cocompact, rigid cellular $G$-action.\footnote{By a \emph{rigid} cellular
action, we mean that the stabilizer of each cell, $e$, acts trivially on $e$.
See \cite{Geo08}.}

All Type F groups are torsion-free, but there are many groups with torsion
that admit $\mathcal{Z}$-structures. To aid in discussing those groups in the
coming sections, we introduce the following definition.

\begin{definition}
A group $G$ is \emph{Type }$F^{\ast}$ if there exists a contractible CW
complex admitting a proper, cocompact, rigid cellular $G$-action. Similarly,
$G$ is \emph{Type }$F_{AR}^{\ast}$ if there exists an AR admitting a proper,
cocompact $G$-action.
\end{definition}

\begin{remark}
\label{Remark: Failure of the West trick}Unfortunately, the trick used above
(showing that Type $F$ = Type F$_{AR}$) relies on a covering space argument
not applicable when $G$ has torsion. Whether there is a group of Type
$F_{AR}^{\ast}$ that is not of Type $F^{\ast}$ is an open question. We will
return to this and related questions in Section \ref{Section: Open questions}.
\end{remark}

There is also room for a definition of \emph{Type} $\underline{F}$, by which
we mean a group admitting a cocompact \underline{$E$}$G$ complex, and
\emph{Type} $\underline{F}_{AR}$, meaning a group that admits a proper
cocompact action on an AR such that stabilizers of finite subgroups are
contractible. These variations are not needed in this paper.

\section{Uniqueness and boundary swapping for c$\mathcal{Z}$%
-structures\label{Section: Uniqueness and boundary swapping for cZ-structures}%
}

We now begin justifying the definitions of the previous section by extending
key theorems about geometric actions, $\mathcal{Z}$-structures, and
$\mathcal{Z}$-boundaries to the realm of proper cobounded coarse near-actions,
c$\mathcal{Z}$-structures, and c$\mathcal{Z}$-boundaries. We start with
generalized versions of three fundamental theorems about a given group $G$.
See \cite{GuMo19} for the classical analogs. For a brief review of the notion of proper homotopy equivalence, see page 302 of \cite{GuMo19}. 

\begin{theorem}
[Coarse uniqueness of geometric models]If a group $G$ admits proper,
cobounded, coarse near-actions on model geometries $\left(  X_{1}%
,d_{1}\right)  $ and $\left(  X_{2},d_{2}\right)  $, then there exists a
continuous coarse equivalence $f:X_{1}\rightarrow X_{2}$. As a consequence,
$X_{1}$ is proper homotopy equivalent to $X_{2}$.
\end{theorem}

\begin{theorem}
[c$\mathcal{Z}$-boundary swapping]\label{Theorem: cZ-boundary swapping}If $G$
admits a c$\mathcal{Z}$-structure based on a model $\mathcal{Z}$-geometry
$\left(  \overline{X},Z,d\right)  $ and $\left(  Y,d^{\prime}\right)  $ is a
model geometry on which there is a proper, cobounded, coarse near-action by
$G$, then there is a model $\mathcal{Z}$-geometry of the form $\left(
\overline{Y},Z,d^{\prime}\right)  $ underlying a c$\mathcal{Z}$-structure on
$G$.
\end{theorem}

\begin{corollary}
[c$\mathcal{Z}$-boundary swapping---alternate version]%
\label{Corollary: cZ-boundary swapping-alternate version}If $G$ admits a
c$\mathcal{Z}$-structure based on a model $\mathcal{Z}$-geometry $\left(
\overline{X},Z,d\right)  $ and $\left(  Y,d^{\prime}\right)  $ is a model
geometry coarsely equivalent to $\left(  X,d\right)  $ or $G$, then there is a
model $\mathcal{Z}$-geometry of the form $\left(  \overline{Y},Z,d^{\prime
}\right)  $ underlying a c$\mathcal{Z}$-structure on $G$.
\end{corollary}

\begin{proof}
Combine Theorem \ref{Theorem: cZ-boundary swapping} with Theorem
\ref{Theorem: Generalized Svarc-Milnor with converse}.
\end{proof}

\begin{corollary}
If $G$ admits a c$\mathcal{Z}$-structure with boundary $Z$ and $G$ is type
F$_{AR}^{\ast}$, then $G$ admits a $\mathcal{Z}$-structure with boundary $Z$.
\end{corollary}

\begin{theorem}
[Shape uniqueness of c$\mathcal{Z}$-boundaries]If $Z_{1}$ and $Z_{2}$ are
c$\mathcal{Z}$-boundaries for a group $G$, then $Z_{1}$ is shape equivalent to
$Z_{2}$.
\end{theorem}

Each of the above theorems is implied by the following collection of more
general theorems, which involve pairs of quasi-isometric groups. This is where
the benefits of our generalization scheme become clear. For example, Corollary
\ref{Corollary: cZ-boundaries are q-i invariants} is significantly cleaner and
more general than its analog for $\mathcal{Z}$-boundaries.

\begin{remark}
Under word length metrics, finitely generated groups are quasi-geodesic
spaces, so there is no difference between quasi-isometric and coarse
equivalent finitely generated groups. For that reason, we stick with the more
common notion of quasi-isometry \textbf{when comparing groups}.
\end{remark}

\begin{theorem}
[Generalized coarse uniqueness of geometric models]%
\label{Theorem: Generalized coarse uniqueness of geometric models}If
quasi-isometric groups $G$ and $H$ admit proper, cobounded, coarse actions on
model geometries $\left(  X_{1},d_{1}\right)  $ and $\left(  X_{2}%
,d_{2}\right)  $, respectively, then there exists a continuous coarse
equivalence $f:X_{1}\rightarrow X_{2}$. In particular, $X_{1}$ is proper
homotopy equivalent to $X_{2}$.
\end{theorem}

\begin{proof}
By Theorem \ref{Theorem: Generalized Svarc-Milnor with converse}, $X_{1}$ and
$X_{2}$ are coarsely equivalent; so we may apply \cite[Cor.5.4]{GuMo19}.
Lemmas \ref{Lemma: uniform contractibility} and
\ref{Lemma: finite macroscopic dimension and refinements} assure the necessary hypotheses.
\end{proof}

\begin{theorem}
[Generalized c$\mathcal{Z}$-boundary swapping ]%
\label{Theorem: Generalized cZ-boundary swapping}Suppose $G$ and $H$ are
quasi-isometric groups; $H$ admits a c$\mathcal{Z}$-structure based on a model
$\mathcal{Z}$-geometry $\left(  \overline{X},Z,d\right)  $; and $\left(
Y,d_{Y}\right)  $ is a model geometry which admits a proper, cobounded, coarse
near-action by $G$. Then there is model $\mathcal{Z}$-geometry of the form
$\left(  \overline{Y},Z,d_{Y}\right)  $ underlying a c$\mathcal{Z}$-structure
on $G$.
\end{theorem}

\begin{proof}
Apply \cite[Th.7.1]{GuMo19}, with Theorem
\ref{Theorem: Generalized coarse uniqueness of geometric models} providing the setup.
\end{proof}

For emphasis, we state the following corollaries, the first of which is a
restatement of Theorem 0.1:

\begin{corollary}
\label{Corollary: cZ-boundaries are q-i invariants}If $G$ and $H$ are
quasi-isometric groups and $H$ admits a c$\mathcal{Z}$-structure, then so does
$G$. Moreover, if $Z$ is a c$\mathcal{Z}$-boundary for $H$, then $Z$ is also a
c$\mathcal{Z}$-boundary for $G$.
\end{corollary}

\begin{corollary}
\label{Corollary: upgrading to a Z-structure}If $G$ is quasi-isometric to a
group $H$ which admits a c$\mathcal{Z}$-structure with boundary $Z$, and $G$
is type F$_{AR}^{\ast}$, then $G$ admits a $\mathcal{Z}$-structure with
boundary $Z$.
\end{corollary}

\begin{example}
If $G$ contains a finite index subgroup $H$ which is CAT(0), i.e., $G$ is
\emph{virtually }CAT(0), it is unclear whether $G$ is also CAT(0) (or even
Type F$^{\text{*}}$). As such, generalized boundary swapping (\cite{GuMo19})
does not guarantee that $G$ admits a $\mathcal{Z}$-structure. However, since
$H\hookrightarrow G$ is a quasi-isometry, Theorem
\ref{Theorem: Generalized cZ-boundary swapping} shows that $G$ admits a
$c\mathcal{Z}$-structure. By similar reasoning, virtual Baumslag-Solitar
groups and virtual systolic groups admit $c\mathcal{Z}$-structures. By
combining this strategy with \cite{Pie18} in an inductive manner, we can
deduce that every poly-(finite or cyclic) group admits a $c\mathcal{Z}%
$-structure. By Corollary \ref{Corollary: upgrading to a Z-structure}, the
moment any of these groups can be shown to act properly and cocompactly on an
AR, the $c\mathcal{Z}$-structure can be upgraded to a $\mathcal{Z}$-structure.
\end{example}

\begin{theorem}
[Uniqueness of coarse-$\mathcal{Z}$-boundaries up to shape]If $Z_{1}$ and
$Z_{2}$ are c$\mathcal{Z}$-boundaries for quasi-isometric groups $G$ and $H$,
then $Z_{1}$ is shape equivalent to $Z_{2}$.
\end{theorem}

\begin{proof}
Let $\left(  \overline{X},Z_{1},d_{X}\right)  $ and $\left(  \overline
{Y},Z_{2},d_{Y}\right)  $ be c$\mathcal{Z}$-structures for $G$ and $H$. By
Theorem \ref{Theorem: Generalized Svarc-Milnor with converse}, $G$ is coarsely
equivalent to $\left(  X,d_{X}\right)  $ so, by hypothesis, $H$ is coarsely
equivalent to $\left(  X,d_{X}\right)  $. Theorem
\ref{Theorem: Generalized Svarc-Milnor with converse} provides a proper
cobounded coarse near-action of $H$ on $X$ so, by Theorem
\ref{Theorem: Generalized cZ-boundary swapping}, there is a model
$\mathcal{Z}$-geometry (underlying a c$\mathcal{Z}$-structure on $H$) of the
form $\left(  \overline{X}^{\prime},Z_{2},d_{X}\right)  $, where $\overline
{X}^{\prime}=X\sqcup Z_{2}$ is a second controlled $\mathcal{Z}$%
-compactification of $X$. From here, the argument provided in \cite[\S 8]%
{GuMo19} goes through unchanged.
\end{proof}

We close this section by placing Theorem
\ref{Theorem: Finite-dimensionality of Z-boundaries} into a group-theoretic
context, where it generalizes results from \cite{Gro87}, \cite{Swe99}, and
\cite{Mor16a}.

\begin{theorem}
Every c$\mathcal{Z}$-boundary of a group $G$ has finite Lebesgue covering dimension.
\end{theorem}

\section{Further applications of c$\mathcal{Z}$%
-structures\label{Section: Applications of cZ-structures}}

In this section, we show that many of the theorems from \cite{BeMe91},
\cite{Bes96}, \cite{Ont05}, \cite{Dra06}, and \cite{GeOn07} remain valid in
the more general context of coarse $\mathcal{Z}$-structures. We begin with
some definitions, notation, and basic facts needed to describe these results.

Aside from some appeals to \emph{group cohomology} (see \cite{Bro82}) and
\emph{coarse cohomology} (see \cite{Roe03}), all cohomology used here is
\v{C}ech-Alexander-Spanier cohomology (see \cite{Mas78}) with coefficients in
a PID $R$. This is a compactly supported cohomology theory which agrees with
classical \v{C}ech cohomology on compact metric spaces. As such, it agrees
with singular cohomology on compact ANRs. For locally compact ANRs it is
isomorphic to singular cohomology with compact supports. For that reason, we
will use the notation $\check{H}^{\ast}\!\left(  \ ;R\right)  $ when applying
this functor to compact metric spaces and $\check{H}_{c}^{\ast}\!\left(
\ ;R\right)  $ for spaces that are not necessarily compact. Since many of the
spaces of interest are non-ANRs, the choice of cohomology theories is crucial.
For relative cohomology, this theory requires subspaces to be closed. A useful
property is that, for every compact metric pair $\left(  Y,A\right)  $,
$\check{H}^{\ast}\!\left(  Y,A;R\right)  \cong\check{H}_{c}^{\ast}\!\left(
Y-A;R\right)  $. More generally, if $\left( Y,A\right)  $ is a closed pair of
locally compact metric spaces, $\check{H}_{c}^{\ast}\!\left(  Y,A;R\right)
\cong\check{H}_{c}^{\ast}\!\left(  Y-A;R\right)  $. As usual, a tilde
indicates reduced cohomology.

\begin{lemma}
\label{Lemma: Cech cohomology vs compact supports}For any model $\mathcal{Z}%
$-geometry $\left(  \overline{X},Z,d\right)  $, $\widetilde{\check{H}}%
^{n}\left(  Z;R\right)  \cong\check{H}_{c}^{n+1}\left(  X; R\right)  $ for all
$n$.
\end{lemma}

\begin{proof}
Since $X=\overline{X}-Z$ is an AR, then so is $\overline{X}$ and hence $\overline{X}$ is contractible.  Then the isomorphism follows from the
above remarks and the exact sequence for the pair $\left(  \overline
{X},Z\right)  $.
\end{proof}

For a locally compact metric space $Y$ the \emph{cohomological dimension }with
respect to $R$ is defined by%
\begin{align*}
\dim_{R}X &  =\max\left\{  n\mid\check{H}_{c}^{n}\left(  U;R\right)
\neq0\text{ for some }U\subseteq_{open}Y\right\}  \\
&  =\max\left\{  n\mid\check{H}_{c}^{n}\left(  Y,A;R\right)  \neq0\text{ for
some }A\subseteq_{closed}Y\right\}
\end{align*}
We let $\dim Y$ denote Lebesgue covering dimension. It is a classical fact
that $\dim_{%
%TCIMACRO{\U{2124} }%
%BeginExpansion
\mathbb{Z}
%EndExpansion
}Y=\dim Y$ whenever the latter is finite \cite{Wal81}. The \emph{global
cohomological dimension }with respect to $R$ of a space $X$ is defined by%
\[
\gcd\nolimits_{R}Y=\max\left\{  n\mid\check{H}_{c}^{n}\left(  Y;R\right)
\neq0\right\}
\]

\subsection{The group cohomology theorem for c$\mathcal{Z}$-boundaries}

One of Bestvina and Mess's initial applications of $\mathcal{Z}$-set
technology \cite{BeMe91} was to reveal a connection between the group
cohomology of a torsion-free hyperbolic group $G$ and topological properties
of its Gromov boundary $\partial G$. A particularly striking assertion is
that
\[
H^{n+1}\left(  G,RG\right)  \cong\widetilde{\check{H}}^{n}\left(  \partial
G;R\right)
\]
for all integers $n$ and PID $R$. In \cite{Bes96}, Bestvina applied the same
reasoning to extend this result to all torsion-free groups admitting a
$\mathcal{Z}$-structure; here the $\mathcal{Z}$-boundary plays the role of
$\partial G$. The following result extends that theorem to c$\mathcal{Z}%
$-boundaries and also allows for groups with torsion. The proof relies on
coarse cohomology of metric spaces as developed by Roe in \cite{Roe03}.

\begin{theorem}
If $G$ admits a c$\mathcal{Z}$-structure based on a model $\mathcal{Z}%
$-geometry $\left(  \overline{X},Z,d\right)  $, then $H^{n+1}\left(
G,RG\right)  \cong\widetilde{\check{H}}^{n}\left(  Z;R\right)  $ for all
integers $n$ and coefficients $R$.
\end{theorem}

\begin{proof}
To avoid confusion, we let $G$ denote the group and $\left\vert G\right\vert $
the corresponding metric space. By \cite[Example 5.21]{Roe03}, $HX^{\ast
}\left(  \left\vert G\right\vert ;R\right)  \cong H^{\ast}\left(  G;RG\right)
$ and, since $\left\vert G\right\vert $ is coarsely equivalent to $\left(
X,d\right)  $, $HX^{\ast}\left(  \left\vert G\right\vert ;R\right)  \cong
HX^{\ast}\left(  X;R\right)  $. By \cite[Theorem 5.28]{Roe03}, $HX^{\ast}\left(
X;R\right)  \cong\check{H}_{c}^{\ast}\left(  X;R\right)  ,$ so an application
of Lemma \ref{Lemma: Cech cohomology vs compact supports} completes the proof.
\end{proof}

\subsection{The Bestvina-Mess formula for c$\mathcal{Z}$-boundaries and model
$\mathcal{Z}$-geometries}

Another key insight from \cite{BeMe91} is that, for a torsion-free hyperbolic
group $G$, $\partial G$ satisfies:%
\begin{equation}
\dim_{R}\partial G=\gcd\nolimits_{R}\partial G\label{Key identity}%
\end{equation}
and since $\partial G$ is always finite-dimensional:%
\begin{equation}
\dim\partial G=\gcd\nolimits_{%
%TCIMACRO{\U{2124} }%
%BeginExpansion
\mathbb{Z}
%EndExpansion
}\partial G\label{Second key identity}%
\end{equation}

\noindent In \cite{Bes96}, Bestvina extended these observation to
$\mathcal{Z}$-boundaries of torsion-free groups admitting a slightly more
restrictive version of $\mathcal{Z}$-structure. Theorem
\ref{Theorem: Finite-dimensionality of Z-boundaries} makes that restriction
unnecessary. Later, Dranishnikov \cite{Dra06} expanded the notion of
$\mathcal{Z}$-structure to allow for groups with torsion and showed that the
analog of (\ref{Key identity}) still holds. Another application of  Theorem
\ref{Theorem: Finite-dimensionality of Z-boundaries} implies
(\ref{Second key identity}). Independently, Geoghegan and Ontaneda
\cite{GeOn07} verified (\ref{Key identity}) and (\ref{Second key identity})
for CAT(0) groups, with the visual boundary standing in for $\partial G$. Here
we extend these results still further.

\begin{theorem}
\label{Theorem: Dranishnikov for cZ-structures}Let $(\overline{X},Z,d)$ be a
c$\mathcal{Z}$-structure on a group $G$. Then for every PID $R$, $\dim
_{R}Z=\gcd_{R}Z$ and, since $Z$ is finite-dimensional, $\dim Z=\gcd_{%
%TCIMACRO{\U{2124} }%
%BeginExpansion
\mathbb{Z}
%EndExpansion
}Z$.
\end{theorem}

\begin{corollary}
\label{Corollary: quasi-isometry invariants}For groups admitting a
c$\mathcal{Z}$-structure, the dimension of the boundary and the cohomological
dimension of that boundary over each PID $R$ are quasi-isometry invariants.
\end{corollary}

\begin{remark}
The point of Corollary \ref{Corollary: quasi-isometry invariants} is that,
since a c$\mathcal{Z}$-structure on $G$ guarantees the existence of a
c$\mathcal{Z}$-structure on all groups quasi-isometric to $G$, these are
invariants of the entire quasi-isometry class of $G.$
\end{remark}

Corollary \ref{Corollary: quasi-isometry invariants} follows from a
combination of Theorem \ref{Theorem: Dranishnikov for cZ-structures}, Lemma
\ref{Lemma: Cech cohomology vs compact supports}, and any one of several
results from Section
\ref{Section: Uniqueness and boundary swapping for cZ-structures}. Theorem
\ref{Theorem: Dranishnikov for cZ-structures} is implied by the following more
general result that does not require a proper coarse near-action.

\begin{theorem}
\label{Theorem: Dranishnikov's theorem for model spaces}Let $(\overline
{X},Z,d)$ be a model $\mathcal{Z}$-geometry. Then for every PID $R$, $\dim
_{R}Z=\gcd_{R}Z$. Since $Z$ is finite-dimensional, we also have $\dim Z=\gcd_{%
%TCIMACRO{\U{2124} }%
%BeginExpansion
\mathbb{Z}
%EndExpansion
}Z$.
\end{theorem}

To obtain the main assertion of Theorem
\ref{Theorem: Dranishnikov's theorem for model spaces} we appeal to the
argument presented in \cite{Dra06}. A little additional care must be taken
since our hypotheses are weaker than his: rather than a geometric action on
$X$, we only assume that $X$ is cocompact. Fortunately, that presents only
minor challenges. A few adjustments are necessary, but mostly it suffices to
reexamine each step of the earlier proof and observe that the weaker
hypothesis is sufficient. As for the adjustments, those were anticipated when
we stated and proved Lemmas \ref{Lemma: uniform contractibility}%
-\ref{Lemma: bounded homotopy}.

In \cite{Dra06}, the reader will come across frequent uses of inclusion
induced homomorphisms $\check{H}_{c}^{n}(V;R)\rightarrow\check{H}_{c}%
^{n}(U;R)$ where $V\subseteq U$ are open sets of a metric space $X$. At first
counterintuitive, since cohomology is a contravariant functor\footnote{To make
cohomology with compact supports a functor, one must restrict to proper
maps.}, these homomorphisms are discussed briefly in \cite[\S 3.3]{Hat02} and
in more detail in \cite[p.216]{GrHa81} and \cite{Mas78}. The following lemma
is a rephrased and slightly generalized version of the key lemma in
Dranishnikov's proof.

\begin{lemma}
[{see \cite[Lemma 3]{Dra06}}]\label{Lemma: Primary Lemma}Let $\left(
X,d\right)  $ be a model geometry and $R$ a PID. Suppose $\check{H}_{c}%
^{i}\left(  X;R\right)  =0$ for all $i>n$. Then there is a number $r$ such
that for every open set $U\subseteq X$, the inclusion induced map $\check
{H}_{c}^{i}(U;R)\rightarrow\check{H}_{c}^{i}(N_{r}(U);R)$ is trivial for all
$i>n$.
\end{lemma}

\noindent The necessary generalization is accomplished by inserting Lemmas
\ref{Lemma: uniform contractibility}-\ref{Lemma: bounded homotopy} in the
appropriate places. From there the proof of Theorem 1 in \cite{Dra06} carries
over to provide a proof of Theorem
\ref{Theorem: Dranishnikov's theorem for model spaces}. The idea is to use
Lemma \ref{Lemma: Primary Lemma} to conclude that any corresponding controlled
$\mathcal{Z}$-boundary would necessarily have dimension $\leq n-1$. Combined
with Lemma \ref{Lemma: Cech cohomology vs compact supports}, this provides the
key inequality, $\dim_{R}Z\leq\gcd_{R}Z$. In the lowest dimension, this
argument provides more than is contained in the statement of Theorem
\ref{Theorem: Dranishnikov's theorem for model spaces}. This difference is
significant since it allows us to obtain extensions of \cite[Corollary 1.3
(d)]{BeMe91} and the $d=0$ case of \cite[Main Theorem]{GeOn07} (which is
already covered by Theorem
\ref{Theorem: Dranishnikov's theorem for model spaces} when $d>0$).

\begin{proposition}
Suppose $(\overline{X},Z,d)$ is a model $\mathcal{Z}$-geometry and $\dim
_{R}Z=0$. Then

\begin{enumerate}
\item $\check{H}_{c}^{1}\!\left(  X;R\right)  \neq0$, and

\item $Z$ is not a one-point set.
\end{enumerate}
\end{proposition}

\begin{proof}
First note that $\check{H}_{c}^{i}\!\left(  X;R\right)  =0$ for all $i>1$,
otherwise $\gcd_{R}Z\geq1$, contradicting the assumption that $\dim_{R}Z=0$.
Since $X$ is connected and noncompact, $\check{H}_{c}^{0}\!\left(  X;R\right)
=0$. So, if $\check{H}_{c}^{1}\!\left(  X;R\right)  =0$, then $\check{H}%
_{c}^{i}\!\left(  X;R\right)  =0$ for all $i$, and the above argument would
imply that $Z=\varnothing$, a contradiction. Assertion 1) follows.

For assertion 2), apply Lemma \ref{Lemma: Cech cohomology vs compact supports}
to conclude that $\widetilde{\check{H}}^{0}\!\left(  Z;R\right)  \neq0$.
\end{proof}

\begin{corollary}
The boundary of a model $\mathcal{Z}$-geometry never has the \v{C}ech
cohomology of a point.
\end{corollary}

\subsection{Generalization of Ontaneda's Almost Geodesic Completeness Theorem}

Recall that a metric space $\left(  X,d\right)  $ is \emph{almost geodesically
complete }if there exists $R>0$ such that, for every $x,y\in X$, there is a
geodesic ray $r:[0,\infty)\rightarrow X$ such that $r\left(  0\right)  =x$ and
$d\left(  r,y\right)  \leq R$. Our next goal is an application of Theorem
\ref{Theorem: Dranishnikov's theorem for model spaces} which can be viewed as
a generalization of the \textquotedblleft Almost Geodesic Completeness
Theorems\textquotedblright\ found in \cite{Ont05} and \cite{GeOn07} for CAT(0)
spaces, and in \cite{BeMe91} (attributed to M. Mihalik) for hyperbolic spaces.

Let $\left(  \overline{X},Z,d\right)  $ be a model $\mathcal{Z}$-geometry and
$H:\overline{X}\times\left[  0,1\right]  \rightarrow\overline{X}$ be a
homotopy that instantly pulls $\overline{X}$ off from $Z$. Using the
contractibility of $X$, we can assume further that $H$ contracts $\overline
{X}$ to a point $x_{0}\in X$. Let
\[
CZ=Z\times\left[  0,\infty\right]  /\sim\text{\ \ and \ \ }OZ=Z\times
\lbrack0,\infty)/\sim
\]
be the \emph{cone} and \emph{open cone} on $Z$, where $\sim$ identifies
$Z\times\left\{  0\right\}  $ to a point. We view $Z$ as a subset $CZ$ by
identifying it with $Z\times\left\{  \infty\right\}  $. By reversing and
reparameterizing $H$, we obtain a map $\overline{F}:\left(  CZ,Z\right)
\rightarrow\left(  \overline{X},Z\right)  $ whose restriction to the open cone
$F:OZ\rightarrow X$ is proper. These maps, and any others with the same
properties, are the topic of the next theorem.

\begin{theorem}
\label{Theorem: Ontaneda for model geometries}Given a model $\mathcal{Z}%
$-geometry $\left(  \overline{X},Z,d\right)  $ and a map $\overline{F}:\left(
CZ,Z\right)  \rightarrow\left(  \overline{X},Z\right)  $ which takes $Z$
identically onto $Z$ and $OZ$ into $X$, there exists $R>0$ such that
$\overline{F}\left(  OZ\right)  $ is $R$-dense in $X$.
\end{theorem}

\begin{proof}
The map $\overline{F}$ induces a commuting diagram of long exact sequences%
\[%
\begin{array}
[c]{ccccccccccc}%
\leftarrow & \check{H}^{k}\left( Z\right)  & \leftarrow &
\check{H}^{k}\left(   \overline{X}\right)  & \leftarrow & \check{H}^{k}\left(
\overline{X},Z\right)  & \leftarrow & \check{H}^{k-1}\left( Z\right)  & \leftarrow & \check{H}^{k-1}\left(  \overline
{X}\right)  & \leftarrow\\
& \cong\ \downarrow\operatorname*{id}  &  & \cong\ \downarrow\ &  &
\downarrow\overline{F}^{\ast} &  & \cong%
\ \downarrow\operatorname*{id}   &  &\cong\ \downarrow\ & \\
\leftarrow &  \check{H}%
^{k}\left(  Z\right) & \leftarrow & \check{H}^{k}\left(  CZ\right)  & \leftarrow & \check{H}^{k}\left(  CZ,Z\right)  &
\leftarrow &\check{H}%
^{k-1}\left(  Z\right)  &  \leftarrow & \check{H}^{k-1}\left(  CZ\right)  & \leftarrow
\end{array}
\]
By the Five-Lemma, $\overline{F}^{\ast}$ is an isomorphism for all $k$. From
this we may conclude that the restriction map $F:OZ\rightarrow X$, which is
necessarily proper, induces isomorphisms $F^{\ast}:\check{H}_{c}^{i}\left(
X\right)  \rightarrow\check{H}_{c}^{i}\left(  OZ\right)  $ for all $i$.

For each integer $i$, let $B_{i}=B_{d}(x_{0},i)$ to obtain an exhaustion
$B_{1}\subseteq B_{2}\subseteq\cdots$ of $X$ by open sets with compact
closures, and let $C_{i}=F^{-1}\left(  B_{i}\right)  $ to get a
corresponding exhaustion $C_{1}\subseteq C_{2}\subseteq\cdots$ of $OZ$. Then
\[
\check{H}_{c}^{k}\left(  X\right)  =\underrightarrow{\lim}\check{H}^{k}\left(
X,X-B_{i}\right)  \text{\ \ and \ \ }\check{H}_{c}^{k}\left(  OZ\right)
=\underrightarrow{\lim}\check{H}^{k}\left(  OZ,OZ-C_{i}\right)
\]
and there is a commuting diagram between direct sequences
\[%
\begin{array}
[c]{ccccccc}%
\rightarrow & \check{H}^{k}\left(  X,X-B_{j-1}\right)  & \rightarrow &
\check{H}^{k}\left(  X,X-B_{j}\right)  & \rightarrow & \check{H}^{k}\left(
X,X-B_{j+1}\right)  & \rightarrow\\
& \downarrow F^{\ast} &  & \downarrow F^{\ast} &  & \downarrow F^{\ast} & \\
\rightarrow & \check{H}^{k}\left(  OZ,OZ-C_{j-1}\right)  & \rightarrow &
\check{H}^{k}\left(  OZ,OZ-C_{j}\right)  & \rightarrow & \check{H}^{k}\left(
OZ,OZ-C_{j+1}\right)  & \rightarrow
\end{array}
\]
By Theorem \ref{Theorem: Dranishnikov's theorem for model spaces}, there exist
$k\geq1$ for which $\check{H}_{c}^{k}\left(  X\right)  $ contains a nontrivial
element $\alpha$, and by properties of direct limits, $j$ can be chosen
sufficiently large that there exists $\alpha_{j}\in\check{H}^{k}\left(
X,X-B_{j}\right)  $ which projects to $\alpha\in\check{H}_{c}^{k}\left(
X\right)  $. Then $F^{\ast}\left(  \alpha_{j}\right)  $ likewise projects to
the nontrivial element $F^{\ast}\left(  \alpha\right)  \in\check{H}_{c}%
^{k}\left(  OZ\right)  $.

Suppose now that our theorem fails. Then, for all $R>0$, there exists
$x_{R}\in X$ such that $B_{d}\left(  x_{R};R\right)  \cap F%
(OZ)=\varnothing.$ By choosing $R>2j$ and using the cocompactness of $X$, we
may choose $g\in\operatorname{Isom}\left(  X\right)  $ such that
$gB_{j}\subseteq B_{d}\left(  x_{R},R\right)  $. For each integer $i$, let
$B_{i}^{\prime}=gB_{i}=B_{d}(gx_{0},i)$ and let $D_{i}=F^{-1}\left(
B_{i}^{\prime}\right)  $ to get a commuting diagram
\[%
\begin{array}
[c]{ccccccc}%
\rightarrow & \check{H}^{k}\left(  X,X-B_{j-1}^{\prime}\right)  & \rightarrow
& \check{H}^{k}\left(  X,X-B_{j}^{\prime}\right)  & \rightarrow & \check
{H}^{k}\left(  X,X-B_{j+1}^{\prime}\right)  & \rightarrow\\
& \downarrow F^{\ast} &  & \downarrow F^{\ast} &  & \downarrow F^{\ast} & \\
\rightarrow & \check{H}^{k}\left(  OZ,OZ-D_{j-1}\right)  & \rightarrow &
\check{H}^{k}\left(  OZ,OZ-D_{j}\right)  & \rightarrow & \check{H}^{k}\left(
OZ,OZ-D_{j+1}\right)  & \rightarrow
\end{array}
\]

Since $g$ is a homeomorphism, $\alpha^{\prime}=(g^{\ast})^{-1}\left(
\alpha\right)  $ is a nontrivial element of $\check{H}_{c}^{k}\left(
X\right)  =\underrightarrow{\lim}\check{H}^{k}\left(  X,X-B_{i}^{\prime
}\right)  $ and $\alpha_{j}^{\prime}=(g^{\ast})^{-1}\left(  \alpha_{j}\right)
$ is an element of $\check{H}_{c}^{k}\left(  X,X-B_{j}^{\prime}\right)  $ that
projects to $\alpha^{\prime}$. By commutativity, $F^{\ast}\left(
\alpha^{\prime}\right)  $ projects to the corresponding nontrivial element of
$\check{H}_{c}^{k}\left(  OZ\right)  =\underrightarrow{\lim}\check{H}%
^{k}\left(  OZ,OZ-D_{i}\right)  $. But, since $B_{j}^{\prime}$ is disjoint
from $F\left(  OZ\right)  $, $D_{j}=\varnothing$, so $\check{H}^{k}\left(
OZ,OZ-D_{j}\right)  =0$, giving a contradiction.
\end{proof}

\begin{remark}
Given $\overline{F}:\left(  CZ,Z\right)  \rightarrow\left(  \overline
{X},Z\right)  $ as above, we can (by restriction) associate to each $z\in Z$
the proper (not necessarily embedded) ray $r_{z}:[0,\infty)\rightarrow X$
which emanates from $x_{0}$ and limits to $z$ in $\overline{X}$. Theorem
\ref{Theorem: Ontaneda for model geometries} can be interpreted as saying that
every $x\in X$ is at a distance $<R$ from one of these rays. By cocompactness
we have:
\end{remark}

\begin{corollary}
\label{Corollary: Ontaneda for model geometries}Given a model $\mathcal{Z}%
$-geometry $\left(  \overline{X},Z,d\right)  $ and a map $\overline{F}:\left(
CZ,Z\right)  \rightarrow\left(  \overline{X},Z\right)  $ which takes $Z$
identically onto $Z$ and $OZ$ into $X$, there exists $S>0$ such that, for
every $x_{1},x_{2}\in X$, there exists $z\in Z$ and $g\in\operatorname{Isom}%
\left(  X\right)  $ such that $d\left(  gx_{0},x_{1}\right)  <S$ and $d\left(
gr_{z},x_{2}\right)  <S$.
\end{corollary}

\begin{proof}
Let $R>0$ be the constant promised in Theorem
\ref{Theorem: Ontaneda for model geometries} and note that for each
$g\in\operatorname{Isom}\left(  X\right)  $, every $x\in X$ is at a distance
$<R$ from a ray $gr_{z}$ for some $z\in Z$. Choose $R^{\prime}>0$ so that
$\left\{  gx_{0}\mid g\in\operatorname{Isom}\left(  X\right)  \right\}  $ is
$R^{\prime}$-dense in $X$ and let $S=\max\left\{  R,R^{\prime}\right\}  $.
\end{proof}

\begin{corollary}
[{after \cite[Corollary 3]{GeOn07}}]Every cocompact proper CAT(0) space
$\left(  X,d\right)  $ is almost geodesically complete.
\end{corollary}

\begin{proof}
Let $\left(  \overline{X},\partial_{\infty}X,d\right)  $ be the model
$\mathcal{Z}$-geometry obtained by attaching to $X$ its visual boundary, and
let the contraction $H:\overline{X}\times\left[  0,1\right]  \rightarrow X$ be
the geodesic retraction to $x_{0}$. Then the rays $r_{z}$ are precisely the
geodesic rays in $X$ emanating from $x_{0}$ and the rays $gr_{z}$ are the
geodesic rays emanating from $gx_{0}$.

Let $S>0$ be the constant promised by Corollary
\ref{Corollary: Ontaneda for model geometries} and let $x,y\in X$. Then there
exists $g\in\operatorname{Isom}\left(  X\right)  $ such that $d\left(
gx_{0},x\right)  <S$ and $d\left(  gr_{z},y\right)  <S$. By a standard
construction in CAT(0) geometry there exists a geodesic ray $r$ emanating from
$x$ and asymptotic to $gr_{z}$. That same construction assures that $d\left(
gr_{z}\left(  t\right)  ,r\left(  t\right)  \right)  \leq S$ for all $t>0$, so
$d\left(  y,r\right)  \leq2S$.
\end{proof}

The nature of Theorem \ref{Theorem: Ontaneda for model geometries} brings to
mind another theorem about geodesic rays in CAT(0) spaces. Geoghegan and
Swenson \cite{GeSw19} (also see their arXiv update which contains slightly
stronger conclusions) showed that a 1-ended proper CAT(0) space $\left(
X,d\right)  $ is semistable (all proper maps $r:[0,\infty)\rightarrow X$ are
properly homotopic) if and only if all \emph{geodesic} rays in $X$ emanating
from a common $x_{0}\in X$ are properly homotopic. After analyzing their
proof, we conclude that we have nothing new to offer, except the observation
that their proof already imples the following generalization. 

\begin{theorem}
Let $\overline{X}=X\sqcup Z$ be a (not necessarily controlled) $\mathcal{Z}%
$-compactification of a (not necessarily cocompact) proper metric 1-ended AR
$\left(  X,d\right)  $. Let $x_{0}\in X$ and $\left\{  r_{z}\right\}  _{z\in
Z}$ be the family of (singular) proper rays described above. Then $X$ is
semistable if and only if $r_{z}$ is properly homotopic to $r_{z^{\prime}}$
for all $z,z^{\prime}\in Z$.
\end{theorem}

\section{$E\mathcal{Z}$-structures and coarse-$E\mathcal{Z}$%
-structures\label{Section EZ-structures and cEZ-structures}}

A discussion of $\mathcal{Z}$-structures would be incomplete without some
mention of $E\mathcal{Z}$-structures. Here we show how the concepts introduced
in this paper can be extended to allow for coarse $E\mathcal{Z}$-structures.

\begin{definition}
Each model \emph{$\mathcal{Z}$}-geometry $\left(  \overline{X},Z,d\right)  $
determines a corresponding \emph{uniform subgroup }of $\operatorname{Isom}%
\left(  X\right)  $ defined by:%
\[
\mathcal{U}\left(  \overline{X},Z,d\right)  =\left\{  \left.  \gamma
\in\operatorname{Isom}\left(  X\right)  \,\right\vert \,\gamma\text{ extends
to a homeomorphism }\overline{\gamma}:\overline{X}\rightarrow\overline
{X}\right\}
\]

\end{definition}

A coarse analog of this is the following.

\begin{definition}
Each model \emph{$\mathcal{Z}$}-geometry $\left(  \overline{X},Z,d\right)  $
determines a corresponding \emph{coarse uniform subset }of
$\operatorname*{Coarse}\left(  X\right)  $ defined by:%

\begin{align*}
c\mathcal{U}\left(  \overline{X},Z,d\right)  =\{ \gamma\in
\operatorname*{Coarse}\left(  X\right)  \,|  &  \ \gamma\text{ extends to a
map }\overline{\gamma}:\overline{X}\rightarrow\overline{X}\\
&  \text{ that is continuous at all points of }Z\}
\end{align*}

\end{definition}

\begin{example}
Recall the model \emph{$\mathcal{Z}$}-geometries $\left(  \overline{%
%TCIMACRO{\U{211d} }%
%BeginExpansion
\mathbb{R}
%EndExpansion
^{2}},S^{1},d_{E}\right)  $ and $\left(  \overline{\mathbb{H}^{2}},S^{1}%
,d_{H}\right)  $ discussed in Example \ref{Example: planar geometries}. In
each case, the corresponding uniform subgroup is the entire isometry group. By
contrast, if we quotient out the upper half-circle in either boundary, we get
a new model \emph{$\mathcal{Z}$}-geometry where the \emph{$\mathcal{Z}$%
}-boundary is still a circle, but now many isometries do not extend. Similar
comments can be made regarding Example \ref{Example: n-dimensional geometries}.
\end{example}

%\begin{example}
%As in the above example, we can obtain model \emph{$\mathcal{Z}$}-geometries
%$\overline{%
%%TCIMACRO{\U{211d} }%
%%BeginExpansion
%\mathbb{R}
%%EndExpansion
%^{n}}$ and $\overline{\mathbb{H}^{n}}$ by adding the visual $\left(
%n-1\right)  $-sphere at infinity to the Euclidean and hyperbolic spaces. If
%$A\subseteq S^{n-1}$ is a non-cellular cell-like set (such as a Fox-Artin arc
%or the Whitehead continuum in $S^{3}$), then quotienting out by $A$ produces
%\emph{$\mathcal{Z}$}-geometries with boundaries not homeomorphic to $S^{n-1}$.
%\end{example}

Recall from Section \ref{Section: Z-structures and coarse Z-structures} that a
$\mathcal{Z}$\emph{-structure} on $G$ is a homomorphism $\phi:G\rightarrow
\operatorname{Isom}\left(  X\right)  $ such that $\ker\phi$ is finite and
$\phi\left(  G\right)  $ is both cocompact and proper. If, in addition,
$\phi\left(  G\right)  \subseteq\mathcal{U}(\overline{X}, Z, d)$, we call this
an $E\mathcal{Z}$\emph{-structure} and $Z$ an $E\mathcal{Z}$\emph{-boundary
for }$G$. We generalize this established definition as follows:

\begin{definition}
\label{Defn: cEZ-structure}A \emph{coarse }$E\mathcal{Z}$\emph{-structure} on
a group $G$ (\emph{c}$E\mathcal{Z}$\emph{-structure} for short) is a
c$\mathcal{Z}$-structure $(\overline{X},Z,d,\psi)$ with the additional
property that, for each $\gamma\in G$, the coarse equivalence $\psi\left(
\gamma\right)  :X\rightarrow X$ extends to a map $\overline{\psi\left(
\gamma\right)  }:\overline{X}\rightarrow\overline{X}$ that is continuous at
all points of $Z$. More succinctly, we require $\psi\left(  G\right)
\subseteq c\mathcal{U}\left(  \overline{X},Z,d\right)  $.
\end{definition}

\begin{remark}\label{Remark:extending boundary action to homeomorphism}
By properness, each $\overline{\psi\left(  \gamma\right)  }$ in the above
definition maps $Z$ into $Z$; moreover, by upgrading the c$\mathcal{Z}%
$-structure to a coarse near-action by continuous maps (see \cite[\S 5]%
{GuMo19}), we can require that $\overline{\psi\left(  \gamma\right)
}:\overline{X}\rightarrow\overline{X}$ be continuous. Lemma 7.4 of
\cite{GuMo19} ensures that every coarse self-equivalence of $X$ that is
boundedly close to $\operatorname*{id}_{X}$ extends (continuously and
uniquely) over $Z$ via the identity. As a result $\left.  \overline
{\psi\left(  1\right)  }\right\vert _{Z}=\operatorname*{id}_{Z}$, and for each
$\gamma\in G$, both $\overline{\psi\left(  \gamma\right)  }\circ\overline
{\psi\left(  \gamma^{-1}\right)  }$ and $\overline{\psi\left(  \gamma
^{-1}\right)  }\circ\overline{\psi\left(  \gamma\right)  }$ are the identity
when restricted to $Z$. It follows that each $\left.  \overline{\psi\left(
\gamma\right)  }\right\vert _{Z}:Z\rightarrow Z$ is a homeomorphism and, with
a little more effort, that restriction gives an actual $G$-action on $Z$. So,
if desired, this requirement could be added to Definition
\ref{Defn: cEZ-structure} without a loss of generality.
\end{remark}

\begin{theorem}
Suppose $G$ admits a c$E\mathcal{Z}$-structure $\left(  \overline
{Y},Z,d\right)  $ and $\left(  X,d^{\prime}\right)  $ is another model
geometry coarsely equivalent to $G$. Then $G$ admits a c$E\mathcal{Z}%
$-structure of the form $\left(  \overline{X},Z,d^{\prime}\right)  $.
\end{theorem}

\begin{proof}
Corollary \ref{Corollary: cZ-boundary swapping-alternate version} ensures that
there is a c$\mathcal{Z}$-structure of the form $(\overline{X},Z,d^{\prime}%
)$for $G$. By Remark \ref{Remark:extending boundary action to homeomorphism}, we can ensure that $\overline{\psi\left(
\gamma\right)  }|_{Z}:Z\rightarrow Z$ is a homeomorphism that gives rise to an
actual $G$-action on $Z$. An application of Proposition 7.5 from \cite{GuMo19}
guarantees that the coarse near action of $G$ on $X$ extends to an action by
homeomorphisms on $Z$.
\end{proof}

\begin{corollary}
Suppose $G$ admits a c$E\mathcal{Z}$-structure. If $G$ is type F$_{\text{AR}%
}^{\text{*}}$ then $G$ admits an $E\mathcal{Z}$-structure. If, in addition,
$G$ is torsion-free, the Novikov Conjecture holds for $G$.
\end{corollary}

\begin{proof}
The first assertion is clear. From there one can make the conclusion regarding
the Novikov Conjecture by applying \cite{FaLa05}, which requires that $G$ be torsion-free.
\end{proof}

\section{Open questions\label{Section: Open questions}}

We close by shining a light on some open questions. To streamline the
discussion, we introduce a few more definitions.

\begin{definition}
\label{Defn. Type Z etc}A group $G$ is of \emph{Type Z }[resp., \emph{Type
$EZ$}] if it admits a $\mathcal{Z}$-structure [resp., $E\mathcal{Z}%
$-structure]. It is of \emph{Type c$Z$} [resp., \emph{Type c$EZ$}] if it
admits a $c\mathcal{Z}$-structure [resp., $cE\mathcal{Z}$-structure]\emph{.}
\end{definition}

The best-know questions about $[E]\!\mathcal{Z}$ structures have been raised
by Bestvina and Farrell-Lafont. We supplement their questions with our own variations.

\begin{question}
Does every Type F group have Type Z? Type EZ? Does every Type $F^{\ast}$ or
Type $F_{AR}^{\ast}$ group have Type Z? Type EZ?
\end{question}

With the benefit of Definition \ref{Defn. Type Z etc}, Theorem
\ref{Theorem: primary goal} can be rephrased as follows: \emph{Type cZ is a
quasi-isometry invariant. }That begs the question.

\begin{question}
Are any of the following quasi-isometry invariants: Type $F^{\ast}$, Type
$F_{AR}^{\ast}$, Type Z, Type EZ, Type cEZ?
\end{question}

As noted in Remark \ref{Remark: Failure of the West trick}, a beautifully
simple question from ANR theory asks:

\begin{question}
Are Type $F^{\ast}$ and $F_{AR}^{\ast}$ equivalent?
\end{question}

\noindent Another purely topological question asks:

\begin{question}
\label{Question: Z-compactifiability}Under what conditions does a model
geometry admit a controlled $\mathcal{Z}$-compactification? Any $\mathcal{Z}$-compactification?
\end{question}

\noindent The reader interested in Question
\ref{Question: Z-compactifiability} might want to look at \cite{ChSi76} and
\cite{Gui01}.\medskip

Finally, the usefulness of $E\mathcal{Z}$-structures in attacks on the Novikov
Conjecture makes the following question natural.

\begin{question}
Does every group $G$ which admits an $E\mathcal{Z}$-structure satisfy the
Novikov Conjecture? [Yes by \cite{FaLa05} when $G$ is torsion-free, and in
some additional cases by \cite{Ros06}; but the general question seems to be
open.]. Does every [torsion-free] group that admits a c$E\mathcal{Z}%
$-structure satisfy the Novikov conjecture?
\end{question}

\bibliographystyle{amsalpha}
\bibliography{Biblio}

\end{document}